\documentclass[a4paper,11pt,reqno]{amsart}
\usepackage{amsmath, enumerate}
\usepackage{amsthm}
\usepackage{graphicx}
\usepackage{amssymb}
\usepackage{appendix}
\usepackage[utf8x]{inputenc}
\usepackage{fancyhdr}
\usepackage{calc}
\usepackage{url}
\usepackage[text={6in,8.6in},centering]{geometry}
\usepackage[italian,english]{babel}
\usepackage[T1]{fontenc}
\usepackage[usenames,dvipsnames]{color}

\makeatletter
\def\cleardoublepage{\clearpage\if@twoside \ifodd\c@page\else%
         \hbox{}%
     \thispagestyle{empty}
     \newpage%
     \if@twocolumn\hbox{}\newpage\fi\fi\fi}
\makeatother

\DeclareMathOperator{\sign}{sign}

\let\cleardoublepage\clearpage

\addto\captionsitalian{}

\newcommand{\R}{\mathbb{R}}
\newcommand{\norm}[1]{\left\lVert#1\right\rVert}

\newtheorem{thm}{Theorem}[section]

\newtheorem{lem}[thm]{Lemma}
\newtheorem{pro}[thm]{Proposition}
\newtheorem{den}[thm]{Definition}
\newtheorem{oss}[thm]{Remark}

\numberwithin{equation}{section}

\begin{document}

\title[An inhomogeneous porous medium equation with large data]{An inhomogeneous porous medium equation \\ with large data: well-posedness}

\author{Matteo Muratori and Troy Petitt}

\address{Matteo Muratori and Troy Petitt: Dipartimento di Matematica, Politecnico di Milano, Piazza Leonardo da Vinci 32, 20133 Milano (Italy)}
\email{matteo.muratori@polimi.it}
\email{troy.petitt@polimi.it}


\makeatletter
\@namedef{subjclassname@2020}{%
	\textup{2020} Mathematics Subject Classification}
\makeatother




\begin{abstract}
	We study solutions of a Euclidean weighted porous medium equation when the weight behaves at spacial infinity like $|x|^{-\gamma}$, for $\gamma\in [0,2)$, and is allowed to be singular at the origin. In particular we show local-in-time existence and uniqueness for a class of large initial data which includes as ``endpoints''  those growing at a rate of $ |x|^{(2-\gamma)/(m-1)}$, in a weighted $L^1$-average sense. We also identify  global-existence and blow-up classes, whose respective forms strongly support the claim that such a growth rate is optimal, at least for positive solutions. As a crucial step in our existence proof we establish a local smoothing effect for large data without resorting to the classical Aronson-B\'{e}nilan inequality and using the B\'enilan-Crandall inequality instead, which may be of independent interest since the latter holds in much more general settings.
\end{abstract}

\maketitle
\tableofcontents

\section{Introduction}
We discuss existence, uniqueness, smoothing effects and blow-up phenomena of very weak solutions of the Cauchy problem for a \emph{weighted porous medium equation} in $ \R^N $, namely
\begin{equation}\label{wpme}
	\begin{cases}
		\rho \, u_t = \Delta\!\left(u^m \right) & \text{in } \mathbb{R}^N\times (0,T) \, , \\
		u  = u_0 & \text{on } \mathbb{R}^N\times \{0\} \, ,
	\end{cases}
\end{equation}
for $ m>1 $ and $T>0$ (possibly equal to $ +\infty $). The weight $\rho \equiv \rho(x) $ is a measurable function satisfying the pointwise bounds
\begin{equation}\label{weight-cond}
k \left( 1 + |x|  \right)^{-\gamma}  \le \rho(x) \le K \left| x \right|^{-\gamma} \qquad \text{for a.e. } x \in \mathbb{R}^N \, , 
\end{equation}
for some $\gamma\in [0,2)$ and ordered constants $ K,k>0$. When dealing with sign-changing solutions, we implicitly set $ u^m := |u|^{m-1}u $ as is standard practice. In the spirit of the work of B\'enilan, Crandall and Pierre \cite{BCP}, where the unweighted equation was investigated, we are interested in considering a wide class of initial data $u_0$, which are assumed to be at least locally integrable with respect to $ \rho $ and can grow at infinity with a suitable power-type rate; as we will see, the interval of existence $[0, T)$ may then depend on the behavior at infinity of the initial datum itself. {As concerns the space dimension, for simplicity we will only state and prove all of our results for $ N \ge 3 $, where a common strategy can be employed. Suitable technical modifications of the proofs also allow to treat the cases $ N=2 $ and $ N=1 $, with however some additional restrictions on $ \rho $: this will be discussed in detail in Remark \ref{N=2}.}

From the physical perspective the inclusion of $\rho\not\equiv1$ represents an inhomogeneity in the density of the porous medium where a certain substance, whose concentration is $u$, diffuses (whereas introducing a weight inside the divergence operator represents an inhomogeneity of the diffusion mechanism, see \cite{GMPo}). This equation (in one dimension) was proposed to model heat transfer in a non-homogeneous medium in \cite{KR}, and was thoroughly studied by Reyes and V\'azquez \cite{RV1,RV2} in higher dimensions for globally integrable data. 

Analyzing the limits of well-posedness of \eqref{wpme} for large initial data is a natural problem and has roots going back {to the pioneering works of Tychonoff and Widder for the classical heat equation \cite{Tyk,Widder1,Widder1-bis, Widder2} (especially in one spacial dimension). Widder showed that positive solutions of the heat equation are uniquely expressed as the convolution of an initial datum with the Gaussian heat kernel, while Tychonoff focused more on sign-changing solutions, proving uniqueness in the optimal growth class $ \exp\!{\left( cx^2 \right)} $, for any $ c>0 $, and exhibiting a famous counterexample of a non-zero solution taking the constant $ 0 $ as an initial datum.} As concerns \emph{nonlocal} diffusion, Barrios, Peral, Soria, Valdinoci \cite{BPSV} and Bonforte, Sire, V\'azquez \cite{BSV} were able to extend the so-called Widder theory to the fractional heat equation. In this case the optimal growth is of power-type, to some extent $ |x|^{2s-\varepsilon} $ for arbitrary $ \varepsilon>0 $.

These results on linear diffusion take advantage of an integral representation of solutions. That is, they exist up to some $ t>0 $ provided the initial datum is integrable with respect to the heat kernel at $t$. This fact offers a clue that for \emph{nonlinear} diffusion equations the optimal class of initial data may have a tight relation to the corresponding ``heat kernel'', namely the solution taking a Dirac delta at $ t=0 $ (also known in our context as the \emph{source} or \emph{Barenblatt} solution), even if clearly representation formulas generally fail. Indeed, the results achieved in this direction so far indicate a coherence with the linear theory: for $\rho\equiv1$ and $m>1$, it was shown in \cite{AC} and \cite{BCP} that the optimal class of initial data can be roughly summarized as those functions growing no faster than $ |x|^{{2}/{(m-1)}}$ (more details below), whereas it is known since \cite{Barenblatt} that for any fixed $t>0$ the source solution has the spacial profile
\[
c_1\left( c_2 - c_3 \, |x|^2 \right)_+^{\frac{1}{m-1}} ,
\]
for suitable positive constants $ \{ c_i \}_{i=1,2,3} $ (note the apparent similarity in the exponents). Furthermore, for \eqref{wpme} with the straight-power weight $ \rho(x) = |x|^{-\gamma} $, $ \gamma \in (0,2) $, the source solution is again explicit and has the spacial form (see \cite{RV1})
\[
c_1\left( c_2 - c_3 \, |x|^{2-\gamma} \right)_+^{\frac{1}{m-1}} .
\]
Comparing with the unweighted equation, this would lead to the hypothesis that the corresponding optimal class of initial data should be of order $ |x|^{{(2-\gamma)}/{(m-1)}}$ as $ |x| \to + \infty $, which we will show is indeed the case. We point out that, for globally integrable initial data, Barenblatt solutions dominate the intermediate asymptotics (we refer to \cite{V-E,RV2}), again in accordance with the linear theory. More recently, the same was shown to be true for the fractional porous medium equation as well (see \cite{Vaz14,GMPmu,GMP-dcds15}). 

Let us briefly describe the main results obtained by \cite{BCP} in the case that $\rho\equiv 1$, by which our weighted well-posedness theory is inspired. There, it is shown that if the initial datum satisfies the average-growth condition
\begin{equation}\label{q52}
	\sup_{R\geq 1} R^{-\frac 2{m-1} -N } \int_{B_R} \left|u_0(x)\right|  dx  < + \infty \, , 
\end{equation} 
then a \emph{very weak} solution $u$ to \eqref{wpme} does exist, at least up to a certain time $ T=T(u_0)>0 $ depending on $u_0$. In addition, for any real parameter $ \alpha $ fulfilling
\begin{equation*}
	\alpha > \frac 1 {m-1} + \frac N 2 \, ,
\end{equation*}
it turns out that $ t \mapsto u(t)$ is a continuous curve down to $t=0$ with values in the space of locally integrable functions $ f $ such that 
\[
\int_{\mathbb R^N} \left|f(x)\right|  \left(1+{|x|^2} \right)^{-\alpha}  dx  < + \infty \, . 
\]
Moreover, they prove that $u$ is the \emph{unique} solution in the following sense: if $v$ is another very weak solution to \eqref{wpme} such that, for every $\epsilon>0$,  
\begin{equation}\label{L-inf}
v \left(1+ |x|^2\right)^{-\frac 1{m-1}} \in L^\infty\!\left(\mathbb R^N\times (\epsilon, T)\right)  , 
\end{equation}
and $ t \mapsto v(t) $ is continuous in $ L^1_{\mathrm{loc}}\!\left( \R^N \right) $, then $v=u$. Exploiting some subtle results from \cite{AC}, it is also observed that the class of initial data that they consider is optimal for the existence of (positive) solutions, since initial traces must necessarily comply with \eqref{q52}. 

We mention that in the proof of existence in \cite{BCP} it is used in a crucial way the fact that  the unique nonnegative solution $u$ of the approximate problem
\[
\begin{cases}
	u_t = \Delta\!\left(u^m\right) & \text{in } \mathbb R^N \times (0,+\infty) \, , \\
	u = u_0 \in L^1\!\left(\mathbb R^N\right) \cap L^\infty\!\left(\mathbb R^N\right)  & \text{on } \mathbb R^N \times \{ 0\} \, ,
\end{cases}
\]
with $ u_0 \ge 0  $, satisfies the so-called \emph{Aronson-B\'enilan} inequality \cite{AB}, that is 
\begin{equation}\label{q54}
	\Delta\!\left( u^{m-1} \right) \geq - \frac {m-1}{m} \, \frac{N}{(m-1)N +2} \, \frac 1 t  \qquad \text{in   } \mathcal{D}' \!\left(\mathbb{R}^N \times (0,+\infty) \right) .
\end{equation}
Using \eqref{q54}, a quantitative local smoothing effect is deduced, which is central in the proof of existence and in fact shows that the constructed solution satisfies \eqref{L-inf}. Unfortunately, estimate \eqref{q54} seems to be strictly related to the Euclidean unweighted setting, and a weighted analogue is not available (it was actually extended to Riemannian manifolds in \cite{VV}, but under nonnegative Ricci curvature). However, it is not difficult to prove that  nonnegative solutions of problem \eqref{wpme}, corresponding to bounded and integrable (w.r.t.~$ \rho $) initial data (namely approximate solutions), satisfy 
\begin{equation}\label{m36}
 \rho u_t = \Delta\!\left( u^m \right) \geq - \frac{\rho u}{(m-1) t} \qquad \text{in  } \mathcal{D}' \!\left(\mathbb{R}^N \times (0,+\infty) \right)  .
\end{equation}
For $ \rho \equiv 1 $ the above estimate is known after \cite{BC}  as \emph{B\'enilan-Crandall} inequality, and its proof essentially relies on a pure time-scaling argument which ensures its validity in greater generality (for instance it holds on Riemannian manifolds regardless of curvature assumptions). Using \eqref{m36}, we are still able to prove that any approximate solution $u$ satisfies a suitable \emph{a priori} local $ L^\infty $-bound entailing that, for every $ \epsilon>0 $, 
\begin{equation*} 
u \left(1+ |x|^{2}\right)^{-\frac {2-\gamma}{2(m-1)}} \in L^\infty\!\left(\mathbb R^N\times (\epsilon, T)\right)  .
\end{equation*}
Such a bound only depends on $ u_0 $ through the quantity 
\begin{equation}\label{L-inf-weight}
 \left\| u_0 \right\|_{1,1} := \sup_{R\geq 1} {R^{-\frac{2-\gamma}{m-1} -N+\gamma }} \int_{B_R} \left| u_0(x) \right| \rho(x) \, dx  \, ,
\end{equation}
which is the correct analogue of \eqref{q52}, and is key to guarantee that the finiteness of $ \left\| u_0 \right\|_{1,1} $ is enough for existence, in accordance with the unweighted theory of \cite{BCP} (see Theorem \ref{Existence}). We find it remarkable that our results can be proved by relying solely on \eqref{m36} and not on the much stronger, and not available here, \emph{Aronson-B\'enilan} inequality. Despite the significant technical challenges, our weighted large-data theory shows consistency with the fact that \emph{slowly decaying} densities, as is $ |x|^{-\gamma} $ when $ \gamma \in (0,2) $,  tend to exhibit a Euclidean-like behavior, which had already been observed in \cite{RV1,RV2} for integrable data. This is not the case when $ \gamma>2 $, see \cite{KRV}. 

{Here the restriction $ \gamma<2 $ appears at several points, the most important of which is the validity of a weighted Sobolev-type inequality, see the proofs of Propositions \ref{pro2} and \ref{l4}. We stress that this is false even in the borderline case $ \gamma=2 $, as only a Poincar\'e-type inequality is available for $ \rho(x) = |x|^{-2} $ (i.e.~Hardy's inequality). Furthermore, note that if $ \gamma \ge  2 $ it is not clear whether a class of ``growing data'' is allowed at all (formula \eqref{L-inf-weight} does not suggest any guess in this sense).}

{The proof of uniqueness is based on a well-established duality method, originally introduced by Hilbert and then adapted to nonlinear diffusion equations of the type treated here in \cite{ACP, Pierre} (see also \cite[Section 6.2]{Vaz07})}. In order to implement such a method in \cite{BCP}  new difficulties, absent in \cite{ACP}, had to be overcome. They were mainly due to the fact that the initial datum $u_0$, and so the solution, belongs to a class of unbounded functions. In our framework (Theorem \ref{thmuniq}), additional technical issues present themselves with respect to \cite{BCP}, due to the lack of regularity and possible singular behavior of $\rho(x)$ as $ |x| \to 0 $. This requires approximating $ \rho $ with smooth sequences and solving dual problems on ``punctured'' balls, carefully estimating normal derivatives. For more details we refer to the proof of Proposition \ref{prop-uniq}, which is at the core of the strategy. 

Although we cannot prove optimality of our set of initial data in the strong form of \cite{AC,BCP}, we show that: 
\begin{itemize}
	\item Any initial datum growing slower than the critical rate leads to \emph{global-in-time} existence (Theorem \ref{Existence}), that is $ T=+\infty $;
	\item Special initial data growing at the critical rate exhibit \emph{blow-up  in a finite time} $T$ (Theorem \ref{blowupthm}), with a sharp bound on $T$.
\end{itemize} 
We believe that this gives very clear evidence that our space of initial data is the best possible, especially keeping in mind the initial-trace optimality result for $\rho\equiv1$ contained in \cite[Theorem 4.1]{AC}. However, their proof strongly relies on a delicate Harnack estimate \cite[Theorem 3.1]{AC} (now referred to as \emph{Aronson-Caffarelli} estimate) which is unknown in the weighted framework, and the techniques of \cite{AC} do not seem applicable. 

\medskip

Finally, we mention that in a second part of this project, which is being carried out in \cite{MPQ2}, we investigate the long-time asymptotics of solutions to \eqref{wpme} corresponding to a certain class of non-integrable initial data, in the spirit of \cite{KU}. To this purpose, the well-posedness theory set up here is of fundamental importance.

\medskip 

The paper is organized as follows. In Section \ref{prelim} we introduce some functional-analytic preliminaries and set notations, we give the precise definition of solution to problem \eqref{wpme}, and then we state our main results. Existence of solutions and key \emph{a priori} estimates are shown in Section \ref{sect-existence}. We devote Section \ref{uniqueness} to the proof of uniqueness via the duality method. In Section \ref{blowup} we analyze a related elliptic equation and prove blow-up results. Finally, in Appendix \ref{global in time} we collect some auxiliary functional facts.

\section{Preliminary material and statements of the main results} \label{prelim}

In this section we first set notations and introduce the basic functional tools needed to address the well-posedness problems, then we state and comment our main results. 

\subsection{Functional setting}\label{sf} 
Throughout, we will make use of some functional spaces with respect to a measurable weight $\rho$ satisfying the two-sided pointwise bound \eqref{weight-cond}
for some $\gamma \in [0,2)$ and constants $ k,K>0 $ with $k<K$, where the dimension $ N $ is assumed to be $ \ge 3 $. In order to define such spaces, for any $ r \ge 1 $ and $ f \in L^1_{\mathrm{loc}}\!\left(\mathbb{R}^N,\rho\right) $ we set
\begin{equation*}\label{m6}
	\left\| f \right\|_{1, r}:= \sup_{R\geq r} {R^{-\frac{2-\gamma}{m-1} -N+\gamma }} \int_{B_R} \left| f(x) \right| \rho(x) \, dx  \, ,
\end{equation*}
where $m>1$, $ B_R $ stands for the (open) ball of radius $ R $ centered at the origin and, for $ p \in [1,\infty) $, 
\begin{equation}\label{lp-weight}
\begin{gathered}
L^p\!\left(\mathbb{R}^N,\rho\right) := \left\{ f \ \text{measurable}: \ \left\| f \right\|_{L^p\left( \mathbb{R}^N , \rho \right)}^p :=  \int_{\mathbb{R}^N} \left| f(x) \right|^p \rho(x) \, dx < +\infty  \right\} ,
\\
L^p_{\mathrm{loc}}\!\left(\mathbb{R}^N,\rho\right) := \left\{ f \ \text{measurable}: \ \int_{B_R} \left| f(x) \right|^p \rho(x) \, dx < +\infty \, , \ \text{for all } R>0 \right\} .
\end{gathered}
\end{equation}
Given any $ f \in L^p_{\mathrm{loc}}\!\left(\mathbb{R}^N,\rho\right) $, $ p \in (1,\infty)$, we will also occasionally employ the following more general $p$-norms:
\begin{equation*}
	\left\| f \right\|_{p, r} := \sup_{R\geq r} R^{-\frac{2-\gamma}{m-1}-\frac{N-\gamma}{p}}\left(\int_{B_R} \left| f(x) \right|^p \rho(x) \, dx\right)^{\frac 1 p} . 
\end{equation*}
Finally, for any $ f \in L^\infty_{\mathrm{loc}}\!\left(\mathbb{R}^N\right) $ we set
\begin{equation}\label{m6a}
	\left\| f \right\|_{\infty, r} := \sup_{R \geq r} \frac{\left\| f \right\|_{L^\infty(B_R)}}{R^{\frac {2-\gamma}{m-1}}} \, .
\end{equation}
Notice that all of the above norms are equivalent with respect to change in $r \ge  1 $ and decrease as $r$ increases. Now we define the associated normed spaces in the typical way, that is 
 \[
X_p := \left\{ f \in L^p_{\mathrm{loc}}\!\left(\mathbb{R}^N,\rho\right) \colon \ \left\| f \right\|_{p,r} < + \infty \right\} 
\]
for $ p \in [1,\infty) $, and  
\[
X_\infty := \left\{ f \in L^\infty_{\mathrm{loc}}\!\left(\mathbb{R}^N\right) \colon \ \left\| f \right\|_{\infty,r} < + \infty \right\} .
\]
By H\"older's inequality and \eqref{weight-cond}, it is straightforward to check that $ X_{q} \hookrightarrow X_{p} $ for all $ 1 \le p < q \le \infty  $. Moreover, each $X_p $ is in fact a Banach space if endowed with any of the $ \| \cdot \|_{p,r} $ norms. In the sequel, for notational convenience, we call $X:=X_{1}$. We will show that initial data belonging to the space $X$ guarantee local well-posedness of solutions for problem \eqref{wpme}. 

Many of the important estimates below will be stated in terms of the auxiliary space $L^1\!\left(\Phi_\alpha\right)$, defined to be the set of all measurable functions $f$ on $\mathbb{R}^N$ such that 
\begin{equation*}
	\left\| f \right\|_{L^1(\Phi_\alpha)} := \int_{\mathbb{R}^N} \left|f(x) \right| \Phi_\alpha(x) \, \rho(x) \, dx <+\infty \, ,
\end{equation*}
where, for any $ \alpha>0 $,
\begin{equation}\label{def-phi-alpha}
\Phi_\alpha(x) := \left( 1+|x|^2 \right)^{-\alpha} \qquad\forall x \in \mathbb{R}^N \, .
\end{equation}
In general there is no relation between $X$ and $L^1\!\left(\Phi_\alpha\right)$, but if $ \alpha $ is sufficiently large it turns out that $ X $ is continuously embedded in $L^1\!\left(\Phi_\alpha\right)$, see Proposition \ref{technical}. This is relevant to our purposes since it is possible to approximate a function in $L^1\!\left(\Phi_\alpha\right)$ by regular and compactly supported functions, but the same is false in $X$ (the corresponding norm is too strong). 

In the estimation of the maximal time of existence up to which a solution to \eqref{wpme} is granted to exist, it will be useful to exploit the following definition:
\[
\ell(f):=\lim_{r\to+\infty} \left\| f \right\|_{1,r} \qquad \forall f \in X \, .
\]
Notice however that $ \ell(\cdot) $ is not a norm because it lacks the positivity property. With this definition in mind, we introduce the following subspace $X_0\subset X$ which, as we will show below (see Appendix \ref{global in time}), is closed and yields global-in-time solutions:
\[
X_0 := \left\{ f \in X\colon \ \ell(f)=0 \right\} .
\]

Next, in the proofs of \emph{a priori} estimates that are crucial for our strategy, it will be convenient to work with a modification of the $\norm{\cdot}_{p,r}$ norms which is obtained by considering appropriate cut-off functions. To this aim, for all $R\geq 1$, let 
\begin{equation}\label{m1}
\phi_R(x):=\phi\!\left(\frac{|x|}{R}\right) \qquad  \forall x \in M \, ,
\end{equation}
where $\phi$ is any real smooth function on $ [0,+\infty) $ satisfying
\begin{equation}\label{m52}
	0 \leq \phi \leq 1 \quad \textrm{in } [0, +\infty) \, , \qquad \phi \equiv 1 \quad \textrm{in } [0,1] \, , \qquad \phi \equiv 0 \quad \textrm{in } [2,+\infty) \, .
\end{equation}
Accordingly, let us set
\begin{equation*}\label{m7}
	\left|f\right|_{p, r}:= \sup_{R\geq r} {R^{-\frac {2-\gamma}{m-1} - \frac{N-\gamma}{p}}} \left(\int_{\mathbb{R}^N} \left|\phi_R  f\right|^p \rho\, dx \right)^{\frac  1 p} \, , \qquad \left| f \right|_{\infty, r} := \sup_{R\geq r} {R^{-\frac {2-\gamma}{m-1}}} \left\| \phi_R  f \right\|_{L^\infty \left(\mathbb{R}^N\right)} ,
\end{equation*}
where $r \ge 1 $ and $p\in[1,\infty)$. It is clear that $ | \cdot |_{p,r} $ and $ \| \cdot \|_{p,r} $ are equivalent norms on $ X_p $, as
\begin{equation}\label{m7-bis}
\left\| f \right\|_{p, r} \le \left|f\right|_{p, r} \le {2^{\left(\frac {2-\gamma}{m-1} + \frac {N-\gamma}{p} \right)}} \left\| f \right\|_{p, r} \qquad \forall f \in X_p \, . 
\end{equation}

When referring to the norms $ |\cdot|_{p,r} $ we will treat $ \phi $ as a	 fixed function once for all, thus avoiding to stress possible dependence of multiplying constants on it.

We point out that the unweighted analogues of the above spaces, on which we modeled the weighted counterparts, were first introduced in \cite{BCP} for the case of $\rho\equiv1$.

Since we will often deal with local or global $ L^p $ spaces, when we write inequalities like $  f \le g $ or identities it will be implicit that they are meant to hold \emph{almost everywhere} (unless otherwise specified). Moreover, in the integrals, explicit dependence on the variables will occasionally be omitted in order to lighten the notation.

\subsection{Exponents and recurrent symbols}\label{exp-not}
First of all, we define two exponents which will appear frequently in the exposition below: 
\begin{equation}\label{m12a}
	\lambda_1 :=\frac{N-\gamma}{(N-\gamma)(m-1)+2-\gamma}\,,
\end{equation}
\begin{equation}\label{m12b}
	\theta :=\frac{2-\gamma}{N-\gamma}\,,
\end{equation}
where $ N \ge 3 $ is the spacial dimension, $ \gamma \in [0,2) $ is related to power-type behavior required in \eqref{weight-cond} and $ m>1 $ accounts for the degeneracy of the differential equation in \eqref{wpme}. Moreover, still in \eqref{weight-cond} we have the privileged constants $ k,K $. 

Also, $c_n>0$ ($ n=1,2,\ldots $) will be unique constants that appear below in the statements of auxiliary results for ``small'' initial data $u_0\in L^1\!\left(\R^N,\rho\right) \cap L^\infty\!\left(\R^N\right)$, whereas $C_n>0$ ($ n=1,2,\ldots $) will be analogous constants appearing in the statements of our main results for ``large'' data $u_0\in X$. On the other hand, $C>0$ will denote a general constant that may change from line to line in calculations, whose precise value is immaterial.

In agreement with \eqref{lp-weight}, Lebesgue spaces with respect to a weight $\rho$, in a domain $ \Omega \subseteq \mathbb{R}^N$, will be written as $L^p(\Omega,\rho)$, where we may add $ _\mathrm{loc}$ as a subscript to indicate a \emph{local} weighted $L^p$ space. The only exception is the aforementioned space $L^1(\Phi_\alpha)$.

Finally, if $ V $ is a Banach space and $ I $ is a real interval, we will write $ C(I;V) $ to denote continuous functions with values in $V$, and $ L^p(I;V) $ to denote measurable functions $f$ with values in $V$ such that
$$
\int_I \left\| f(t) \right\|_V^p dt <+\infty \, , 
$$ 
with obvious adaptation to $ L^\infty(I;V) $. Moreover, $ W^{1,1}(I;V) $ will stand for the space of functions $ f \in L^1(I;V) $ such that also $ f_t \in L^1(I;V) $ (see \cite[Section 2]{BGa}). In the special case where $ V $ is a local Lebesgue space, e.g.~$ L^1_{\mathrm{loc}}\!\left( \mathbb{R}^N,\rho \right) $, this will mean that the same properties hold in $  B_R$ for every $ R>0 $. If $ V $ is a space of functions, e.g.~$ L^2\!\left( \R^N , \rho  \right) $, with abuse of notation we will often write $ f(t) $ in place of $ f(\cdot,t) $. 


\subsection{Existence and uniqueness results}\label{sect: exuni}

We now provide a suitable notion of solution to~\eqref{wpme}, and then establish well-posedness results for initial data and solutions belonging to the space $ X $.

\begin{den}\label{defsol}
	Let  $N\geq3$, $m>1$, $ T \in (0,+\infty ] $ and $\rho$ be a measurable function satisfying \eqref{weight-cond} with respect to some $\gamma\in[0,2)$ and $k,K>0$. Let $u_0\in L^1_{\mathrm{loc}}\!\left(\mathbb{R}^N,\rho\right)$. Then we say that a function $ u $ is a solution of problem \eqref{wpme} if 
	$$
	u \in C\!\left([0,T);L^1_{\mathrm{loc}}\!\left(\mathbb{R}^N,\rho\right)\right) , \quad u^m \in L^1_{\mathrm{loc}}\!\left( \mathbb{R}^N \times (0,T) \right) , \quad u(0) = u_0
	$$
 and 
	\begin{equation}\label{q50}
		-\int_0^T \int_{\mathbb{R}^N} u \, \phi_t \, \rho  \, dx dt = \int_0^T \int_{\mathbb{R}^N} u^m \, \Delta \phi \, dx dt
	\end{equation}
	for all $\phi\in C^\infty_c\!\left(\mathbb{R}^N\times (0, T)\right)$.
\end{den}

In the following, solutions to \eqref{wpme} will tacitly be understood in the sense described above. Note that, except for the requirement of continuity of $ u(t)$ as a function with values in $ L^1_{\mathrm{loc}}\!\left( \mathbb{R}^N , \rho \right) $,  Definition \ref{defsol} yields the typical concept of \emph{very weak} solution (or a solution in the sense of distributions). 

\smallskip 

The next result establishes existence of solutions to \eqref{wpme} for initial data $u_0$ belonging to $X$, and gives some key additional smoothing and stability estimates.

\begin{thm}[Existence]\label{Existence}
	Let  $N\geq3$, $m>1$ and $\rho$ be a measurable function satisfying \eqref{weight-cond} with respect to some $\gamma\in[0,2)$ and $k,K>0$. Let $u_0 \in X $. Then there exists a solution $ u $ of problem \eqref{wpme}  with $ T=T(u_0) $, where 
	\begin{equation}\label{def-t-u0}
		T(u_0) := \frac{C_1}{\left[ \ell(u_0) \right]^{m-1}} \quad \text{if } u_0 \in X \setminus X_0 \, , \qquad  T(u_0) := + \infty \quad \text{if } u_0 \in X_0 \, , 
	\end{equation}
	for a suitable positive constant $ C_1 $ depending only on $ N,m,\gamma,k,K $. We refer to $ u $ as the constructed solution.
	
	Furthermore, setting
	\begin{equation*}\label{maxtime}
		T_r(u_0) := \frac{C_1}{\left\| u_0 \right\|_{1,r}^{m-1}} \, , \qquad r \ge 1 \, ,
	\end{equation*}
	there exist positive constants $C_2, C_3$ depending only on $ N,m,\gamma,k,K $ such that 
	\begin{equation}\label{smoothing estimate}
		\left\| u(t) \right\|_{1,r} \leq C_2 \left\| u_0 \right\|_{1,r}  \qquad \forall t \in \left( 0 , T_r(u_0) \right) ,
	\end{equation}
	\begin{equation}\label{key estimate}
		\left\| u(t) \right\|_{\infty,r} \leq C_3 \, t^{-\lambda_1} \left\| u_0 \right\|_{1,r}^{ \theta \lambda_1} \qquad \forall t \in \left( 0 , T_r(u_0) \right) ,
	\end{equation}
	and $ u \in C\!\left( \left[ 0, T(u_0) \right) ; L^1\!\left( \Phi_\alpha \right) \right) $ provided
	\begin{equation}\label{cond-alpha}
	\alpha> \frac{2-\gamma}{2(m-1)}+\frac{N-\gamma}{2} \, . 
	\end{equation}
	
	We also have the following $L^1$-contraction results and ordering principle. If $ v $ is the constructed solution associated with another initial datum $ v_0 \in X $,  there exist a positive constant $  C_4$ depending only on $ N,m,\gamma,k,K,r,\alpha,\| u_0 \|_{1,r} , \| v_0 \|_{1,r} $ such that
		\begin{equation}\label{dependence on data 1}
	\left\| u(t)-v(t) \right\|_{L^1(\Phi_\alpha)} \leq \exp\!\left({C_4 \, t^{\theta\lambda_1}}\right) \left\| u_0-v_0 \right\|_{L^1(\Phi_\alpha)} \qquad \forall t \in \left( 0 , T_r(u_0) \wedge T_r(v_0) \right) , 
	\end{equation}
	and a positive constant $ C_5 $ depending only on $ N,m,\gamma,k,K, \| u_0 \|_{1,r} , \| v_0 \|_{1,r} $ such that
	\begin{equation}\label{dependence on data 2}
	\left| u(t)-v(t) \right|_{1,r} \leq \exp\!\left({C_5 \, t^{\theta \lambda_1}}\right) \left| u_0-v_0 \right|_{1,r} \qquad \forall t \in \left( 0 , T_r(u_0) \wedge T_r(v_0) \right) .
	\end{equation}
	Moreover, $u_0\leq v_0$ implies $u(t) \leq v(t)$ for all $  t \in \left( 0 , T(u_0) \wedge T(v_0) \right) $.   
	
	Finally, if in addition $ u_0 \in X_0 $ then $ u \in  C\!\left([0,+\infty) ; X \right) $ and 
	\begin{equation}\label{slow}
    \underset{|x| \to +\infty}{\operatorname{ess}\lim} \ |x|^{-\frac{2-\gamma}{m-1}}  \,  u(x,t) = 0 \qquad \forall t > 0 \, .
	\end{equation} 
\end{thm}

Under suitable assumptions, which are fulfilled by the constructed solutions of Theorem \ref{Existence}, we can also establish uniqueness of solutions whose initial data belong to $X$.
\begin{thm}[Uniqueness]\label{thmuniq} 
	Let  $N\geq3$, $m>1$, $ T \in (0,+\infty] $ and $\rho$ be a measurable function satisfying \eqref{weight-cond} with respect to some $\gamma\in[0,2)$ and $k,K>0$. Let $ u $ and $ v $ be any two solutions of problem \eqref{wpme}, corresponding to the same initial datum $u_0\in X$, such that
	\begin{equation}\label{thmuniq-hp} 
	u,v \in L^\infty_{\mathrm{loc}}((0,T);X_\infty) \, , \quad u,v \in L^\infty_{\mathrm{loc}}([0,T);X) \, .
	\end{equation}
	Then $u = v$. 
\end{thm}

\subsection{Blow-up results}

In the case of the straight power $\rho(x)=|x|^{-\gamma}$ for $\gamma\in(0,2)$, there are explicit solutions to \eqref{wpme} that exhibit a blow-up behavior at arbitrarily small times, everywhere in space. More precisely, given an initial datum 
\begin{equation}\label{explicit solution data}
u_0(x)=\left(a+b\,|x|^{2-\gamma}\right)^{\frac{1}{m-1}}  ,
\end{equation}
for any parameters $a \ge 0 $ and $ b>0$, elementary computations show that the function 
\begin{equation}\label{explicit solution}
	u(x,t)=\left[ \frac{a \, T^\kappa}{(T-t)^\kappa}+\frac{\kappa}{m\,(2-\gamma)(N-\gamma)}\frac{|x|^{2-\gamma}}{(T-t)}\right]^{\frac{1}{m-1}} 
\end{equation}
is an explicit solution of problem \eqref{wpme}, provided
$$
\kappa=\lambda_1 \, (m-1) \, , \qquad T=\frac{\kappa}{m\,(2-\gamma)(N-\gamma)\,b} \, . 
$$
It is straightforward to check that it fulfills all of the requirements of Theorem \ref{thmuniq}, so we can assert that such a function is \emph{the} solution of \eqref{wpme}, which clearly blows up at $t=T$ for all $x\in\mathbb{R}^N$.  

\smallskip 
Our goal here is to find a suitable class of data, comparable to \eqref{explicit solution data}, for which the corresponding solutions blow up in finite time, for general weights satisfying \eqref{weight-cond}. We cannot expect any longer explicit profiles as in \eqref{explicit solution} that flow from a two-parameter family of initial data. Nevertheless, starting from the simple observation that for $ a=0 $ they give rise to \emph{separable solutions} we are able to prove a blow-up result that is valid for generic, strictly positive initial data growing at the critical rate and bounded above and below by certain functions to be made precise here, under the additional requirement that $ \rho $ is \emph{radial}. To this purpose, a crucial role will be played by positive (radial) solutions to the semilinear elliptic equation
\begin{equation}\label{ellip-base}
\Delta\!\left(w^m\right) = \rho \, w \qquad \text{in } \R^N  .
\end{equation}
The latter should in general be understood in a weak sense, since  $ \rho $ is \emph{a priori} not regular and can be singular at the origin. Nevertheless, the low degree of the singularity entailed by \eqref{weight-cond} and $ \gamma < 2 $ ensures that weak solutions actually belong to some $ W^{2,p}_{\mathrm{loc}}\!\left( \R^N \right)$ space for a small enough $ p>1 $, so they are in fact \emph{strong} solutions (but not classical). Below we will simply refer to them as ``solutions'', having in mind these properties. 

\smallskip 
The next proposition, dealing with \eqref{ellip-base}, is an adaptation of \cite[Lemma 5.5]{GMPjmpa}, in a certain sense simpler because we are working in $\mathbb{R}^N$ and on the other hand complicated by the inclusion of the possibly singular weight $\rho$. 

\begin{pro}\label{sol-ellip}
Let  $N\geq3$, $m>1$, $ T \in (0,+\infty) $ and $\rho$ be a {radial} measurable function satisfying \eqref{weight-cond} with respect to some $\gamma\in[0,2)$ and $k,K>0$. Then there exists a family $ \{ W_{\beta} \}_{\beta>0} \subset X_\infty $ of positive radial solutions to the semilinear elliptic equation 
\begin{equation}\label{ellip-ord}
\Delta \! \left( W_\beta^m \right)  = \frac{\rho}{T(m-1)} \, W_\beta \qquad \text{in } \R^N 
\end{equation}
such that $ W_\beta(0)=\beta $, $ W_{\beta_1} < W_{\beta_2} $ if $ \beta_1 < \beta_2 $ and 
\begin{equation}\label{ellip-ord-2} 
\frac{\underline{C}}{\left[\ell\!\left(W_\beta\right)\right]^{m-1}} \le T  \le \frac{\overline{C}}{\left[\ell\!\left(W_\beta\right)\right]^{m-1}}
\end{equation}
for some positive constants $ \underline{C} , \overline{C} $ depending only on $ N,m,\gamma,k,K $.
\end{pro}

By a straightforward computation one sees that the separable function 
\begin{equation}\label{def u2}
	{U}_{\beta}(x,t):=\left(1-\frac{t}{T}\right)^{-\frac{1}{m-1}} W_\beta (x)
\end{equation}
is a solution of \eqref{wpme} with initial datum $ u_0 = W_\beta $, which clearly blows up everywhere at $ t=T $. Due to \eqref{ellip-ord-2}, this shows to some extent the optimality of the existence time identified in \eqref{def-t-u0}, because such solutions cannot be extended beyond that threshold (up to constants). Using the fact that the family $ \{ W_\beta \} $ is ordered, we can prove that the class of initial data for which blow-up occurs is actually larger. 

\begin{thm}[Blow-up]\label{blowupthm}
Let  $N\geq3$, $m>1$, $ T \in (0,+\infty) $ and $\rho$ be a {radial} measurable function satisfying \eqref{weight-cond} with respect to some $\gamma\in[0,2)$ and $k,K>0$. Let $ u_0 \in X $ fulfill
	\begin{equation}\label{blowupthmeq}
W_{\beta_1}(x) \le  u_0(x) \le  W_{\beta_2}(x) \qquad \text{for a.e. } x \in \R^N 
	\end{equation}
for some $ \beta_2 > \beta_1 >0 $. Then the corresponding solution $u$ of problem \eqref{wpme} blows up pointwise at $t=T$, in the sense that
\begin{equation}\label{ptwse-blowup}
\underset{t \to T^-}{\operatorname{ess}\lim} \ u(x,t) = + \infty \qquad \text{for a.e. } x \in \R^N \, . 
\end{equation}
Moreover, we have that
\begin{equation}\label{ellip-ord-3} 
\frac{\underline{C}}{\left[\ell(u_0)\right]^{m-1}} \le T  \le \frac{\overline{C}}{\left[\ell(u_0)\right]^{m-1}} \, .
\end{equation}
\end{thm}

\begin{oss}\rm \label{nonexi}
If $ u_0 $ is a locally bounded and nonnegative initial datum having at infinity a ``supercritical'' growth, that is  
$$
\lim_{|x|\to+\infty} \frac{u_0(x)}{|x|^{\frac{2-\gamma}{m-1}}} = + \infty \, ,
$$
then it is possible to show that no nonnegative very weak solution of problem \eqref{wpme} exists, in the spirit of \cite[Corollary 2.6]{GMPjmpa}. The proof, which we omit, is analogous to the one provided in such paper; essentially, it takes advantage of suitable barriers that can be constructed by reasoning as in Lemma \ref{bounded-data} and it strongly relies on our uniqueness results (i.e.~Theorem \ref{thmuniq} together with Remark \ref{comparison}).
\end{oss}

\begin{oss}\rm \label{opt-uni}
The issue of ``strong'' optimality for the class of data $X$, at least in the case of nonnegative solutions, remains open. By this, we mean showing that the sole existence of a nonnegative solution to the differential equation in \eqref{wpme} implies that the corresponding initial trace $ u_0 $ is well defined and satisfies $ \| u_0 \|_{1,r} < +\infty $. For $ \rho \equiv 1 $, it was proved by Aronson and Caffarelli in \cite[Theorem 4.1]{AC}, up to the fact that the initial trace is in general a \emph{Radon measure} (in this regard a result similar to \cite[Proposition 1.6]{BCP} could also be proved in our weighted setting). However, the techniques employed there do not seem applicable in the presence of a weight. Nevertheless, Theorem \ref{blowupthm} and Remark \ref{nonexi} suggest that such an optimality result is likely to be true also under \eqref{weight-cond}.
\end{oss}


\section{Existence}\label{sect-existence}

In the rest of the paper, to avoid repetitions, we will implicitly assume that $N\geq3$, $m>1$ and $\rho$ is a measurable function satisfying \eqref{weight-cond} with respect to some $\gamma\in[0,2)$ and $k,K>0$. 

\smallskip 

In order to prove Theorem \ref{Existence}, we follow a method originally developed in \cite{BCP} for the case of $ \rho  \equiv 1 $, which however requires nontrivial adaptations to treat a weight $ \rho $ as above.

\subsection{Outline of the strategy}
First of all, for every $n\in \mathbb N$ and $ u_0 \in X $, we consider the approximate problem
\begin{equation}\label{q10}
	\begin{cases}
		\rho u_t  =  \Delta \! \left( u^m \right) & \text{in } \mathbb{R}^N\times \mathbb{R}^+ \, , \\
		u = u_{0n} & \text{on } \mathbb{R}^N\times \{ 0 \} \, ,
	\end{cases}
\end{equation}
where $\{u_{0n}\} \subset L^1\!\left(\mathbb{R}^N,\rho\right) \cap L^\infty\!\left(\mathbb{R}^N\right)$ is a suitable sequence of initial data such that 
\begin{equation*}\label{m65}
\left|	u_{0n} \right| \leq \left| u_0 \right|  \quad  \forall  n\in \mathbb N  \, , \qquad \lim_{n \to \infty} u_{0n} = u_0 \quad \text{a.e.~in  } \mathbb{R}^N \, .    
\end{equation*}
Existence and uniqueness of a weak \emph{energy} solution $u_n$ of problem \eqref{q10} can be easily obtained by using the techniques of \cite{RV1, GMPo} (see Definition \ref{def2} and Proposition \ref{pro1}), along with an ordering principle for ordered initial data. Additionally, we can prove some \emph{a priori} estimates involving the ``large''  norms $\norm{\cdot}_{1,r}$, $\norm{\cdot}_{\infty,r}$  and stability bounds in $ X $ and in the weighted space $ L^1 (\Phi_\alpha )$, which ensure continuous dependence with respect to the initial data (see Propositions \ref{l4} and \ref{continuous dependence estimates}). Such estimates will be crucial in order to pass to the limit in \eqref{q10} and show that a solution to \eqref{wpme}, in the sense of Definition \ref{defsol}, does exist. 

\smallskip 
Before entering the most technical parts of this section, we recall the notion of weak energy solution to \eqref{wpme} for an initial datum in $L^1\!\left(\mathbb{R}^N,\rho\right) \cap L^\infty\!\left(\mathbb{R}^N\right)$, which is by now standard. 
\begin{den}\label{def2}
Let $u_0\in L^1\!\left(\mathbb{R}^N,\rho\right) \cap L^\infty\!\left(\mathbb{R}^N\right)$. Then we say that a function $ u $ is a weak energy solution of problem \eqref{wpme}, with $T=+\infty $, if 
	\begin{equation}\label{energy}
	u \in  C\!\left([0,+\infty); L^1\!\left(\mathbb{R}^N,\rho \right) \right)\cap L^\infty\!\left( \mathbb{R}^N \times (0 , +\infty) \right) , \quad \nabla u^m \in L^2_{\mathrm{loc}} \!\left((0,+\infty); L^2\!\left(\mathbb{R}^N \right) \right) ,  
	\end{equation}
$ u(0)=u_0 $ and 
	\begin{equation*}\label{energy-formulation}
	\int_0^{+\infty} \int_{\mathbb{R}^N} u\, \phi_t \, \rho \, dx dt  = \int_0^{+\infty} \int_{\mathbb{R}^N}  \left\langle \nabla u^m , \nabla \phi \right\rangle  dx dt
	\end{equation*}
	for all $\phi\in C^\infty_c\!\left( \mathbb{R}^N\times (0, +\infty) \right)$. 
\end{den}

In Subsection \ref{exiapp} we will see that the approximate problems are well posed, and the corresponding solutions enjoy good regularity properties that allow us to justify all the computations we need.   

\subsection{A key global elliptic estimate}\label{ke}

Here we state and prove a technical estimate we need in order to establish existence and the $ X $-$X_\infty$ smoothing effect. The proof is inspired from \cite[Proposition 1.3 on p.~59]{BCP}, but some nontrivial adaptations are carried out because we make a weaker hypothesis, namely \eqref{m10}, which is compatible with the B\'enilan-Crandall inequality \eqref{weak BC estimate}. This is crucial for us since in the weighted framework the analogue of the Aronson-B\'enilan inequality \eqref{q54} is in general not available, therefore we cannot require a bound from below on $ \Delta u^{m-1} $.  

\begin{pro}\label{pro2}
Let $u \in L^\infty\!\left(\mathbb{R}^N\right)$, with $ u\geq 0$. Suppose moreover that $ \nabla u^m \in L^2_{\mathrm{loc}}\!\left( \mathbb{R}^N \right) $, $ \Delta u^m \in L^1_{\mathrm{loc}}\!\left(\mathbb{R}^N\right) $ and
	\begin{equation}\label{m10} 
		\Delta u^{m} \geq - \Lambda \, \rho u  
	\end{equation}
	for some constant $\Lambda> 0$. Then there exists a positive constant $c_0$, depending only on $N,m,\gamma,k,K$, such that for all $ r \ge 1 $ we have
	\begin{equation}\label{m11}
		\left\| u \right\|_{\infty, r}^{m-1}  \leq c_0 \left( \Lambda^{\lambda_1(m-1)} \left\| u \right\|_{1,r}^{\theta \lambda_1(m-1)} + \left\| u \right\|_{1, r}^{m-1}\right).
	\end{equation}
\end{pro}
\begin{proof}
	Let $\psi\in C^\infty_c\!\left(\mathbb{R}^N\right)$, with $ \psi\geq 0$ and $ \psi^m $ at least $ C^2\!\left(\mathbb{R}^N\right) $. Thanks to the assumptions on $u$, the product formula
	\[  
	\Delta \left( \psi u \right)^m  = \psi^m \, \Delta u^m   + 2 \left\langle \nabla \psi^m , \nabla u^m \right\rangle + u^{m} \, \Delta \psi^m 	\]
holds. Hence, by using \eqref{m10}, we infer that
	\begin{equation}\label{m14}
		\Delta \left(\psi u \right)^m \geq - \Lambda \, \psi^m \, \rho u + 2 \left\langle \nabla \psi^m , \nabla u^m \right\rangle + u^{m} \, \Delta \psi^m \, .
	\end{equation}
Given any $ q \ge m $, if we multiply \eqref{m14} by $\left(\psi u\right)^q$ and integrate over $\mathbb{R}^N$ we obtain 
	\begin{equation}\label{m14a}
		\begin{aligned}
			& \, \int_{\mathbb{R}^N} \left\langle \nabla\left(\psi u\right)^q , \nabla \left(\psi u\right)^m \right\rangle dx \\
			\leq  & \, \Lambda \int_{\mathbb{R}^N} \psi^{q+m} \, u^{q+1} \, \rho \, dx - 2 \int_{\mathbb{R}^N} \left( \psi u \right)^q \left\langle \nabla \psi^m , \nabla u^m \right\rangle dx - \int_{\mathbb{R}^N} \psi^q \, u^{q+m} \, \Delta \psi^{m} \, dx \, .
		\end{aligned}
	\end{equation}
Note that this computation is justified since $\left(\psi u\right)^q \in H^1_c\!\left( \mathbb{R}^N \right) $, the symbol $ H^1_c\!\left( \mathbb{R}^N \right) $ denoting the space of $ H^1\!\left( \mathbb{R}^N \right) $ functions with compact support. Now we observe that the LHS of \eqref{m14a} can be rewritten as
	\begin{equation}\label{m15}
		\int_{\mathbb{R}^N} \left \langle \nabla \left(\psi u\right)^q , \nabla \left(\psi u\right)^m \right\rangle dx = \frac{4 m q}{(m+q)^2} \, \int_{\mathbb{R}^N} \left| \nabla \left( \psi u \right)^{\frac{q+m}2} \right|^2 dx \, .
	\end{equation}
	Furthermore, an integration by parts of the middle term in the RHS of \eqref{m14a}, plus some algebraic manipulations, yield
	\begin{equation}\label{m16}
		\begin{aligned}
			& \, - 2 \int_{\mathbb{R}^N} \left( \psi u \right)^q \left\langle \nabla \psi^m , \nabla u^m \right\rangle dx - \int_{\mathbb{R}^N} \psi^q \, u^{q+m} \, \Delta \psi^{m} \, dx  \\
			= & \,   -\frac{2m^2}{(m+q)^2} \int_{\mathbb{R}^N} \left\langle \nabla \psi^{q+m} , \nabla u^{q+m} \right\rangle - \int_{\mathbb{R}^N} \psi^q \, u^{q+m} \, \Delta \psi^{m} \, dx \\
			= & \,  \frac{m}{m+q} \int_{\mathbb{R}^N} u^m \left(u\psi\right)^q \left[ \left(m^2+mq-m+q\right) \psi^{m-2} \left| \nabla \psi \right|^2 +(m-q) \, \psi^{m-1} \, \Delta\psi\right] dx \, . 
		\end{aligned}
	\end{equation}
Again, the validity of the computation is guaranteed by the local regularity properties of $ \psi^m $ and $ u^m $. For each $ R \ge 1 $, let us pick $\psi =\phi_R$, where $\phi_R$ is defined in \eqref{m1}. It is direct to check that one can choose $\phi$ in such a way that \eqref{m52} holds, $ \phi^m \in C^2([0,+\infty))  $ and 
	\begin{equation}\label{m17}
	\phi(r)^{m-1} \left[ \left| \phi''(r) \right| + \frac{N-1}{r} \left| \phi'(r)\right| \right]  + \phi(r)^{m-2} \left|\phi'(r)\right|^2  \leq C \, \phi(r) \, \chi_{[1,2]}(r) \qquad \forall r > 0 \, ,
	\end{equation} 
	for some positive constant $C$ depending only on $ N $ and $m$. Hence, in view of the definition of $ \phi_R $ and \eqref{m17}, it follows that 
	\begin{equation}\label{m18} 
		\phi_R(x)^{m-1} \left|\Delta \phi_R(x) \right| + \phi_R(x)^{m-2} \left|\nabla \phi_R(x)\right|^2 \leq \frac{C}{R^2} \, \phi_R(x) \, \chi_{[R,2R]}(|x|) \qquad \forall x \in \mathbb{R}^N \, .
	\end{equation}
	From here on we will let $C $ denote a general positive constant depending only on $ N $, $m$ and possibly on $ \gamma $, $k$, $K$, which may change from line to line. By combining \eqref{m14a}, \eqref{m15}, \eqref{m16}, \eqref{m18} and recalling that $0\leq \phi_R \leq 1$, we infer the bound 
	\begin{equation*}\label{m19}
		\int_{\mathbb{R}^N} \left|\nabla\left(\phi_R u \right)^{\frac{q+m}2}\right|^2 dx \leq
		C \, q \left[\Lambda \, \int_{\mathbb{R}^N} \left(\phi_R u \right)^{q+1} \rho \, dx + \frac 1{R^2} \, \int_{B_{2R} \setminus B_R} \left(\phi_R u \right)^{q+1} u^{m-1} \, dx \right] ,
	\end{equation*}
	which, given assumption \eqref{weight-cond} on the weight (here we need the leftmost inequality), entails 
	\begin{equation}\label{m19-bis}
			 \int_{\mathbb{R}^N} \left|\nabla\left(\phi_R u \right)^{\frac{q+m}2}\right|^2 dx
			\leq  C \, q\left[\Lambda \int_{\mathbb{R}^N} \left(\phi_R u \right)^{q+1} \rho \, dx + \frac 1{R^{2-\gamma}} \int_{B_{2R} \setminus B_R} \left(\phi_R u \right)^{q+1} u^{m-1} \, \rho \, dx \right] ,
	\end{equation}
	whence, in view of definition \eqref{m6a}, it follows that
	\begin{equation}\label{m20}
		\int_{\mathbb{R}^N} \left|\nabla\left( \phi_R u \right)^{\frac{q+m}2}\right|^2 dx
		\leq C \, q \left(\Lambda + \left\| u \right\|_{\infty, r}^{m-1} \right) \int_{\mathbb{R}^N} \left( \phi_R u \right)^{q+1} \rho \, dx \, .
	\end{equation}
	Still as a consequence of \eqref{weight-cond} (in this case we need the rightmost inequality), we know that the weighted Sobolev inequality
	\begin{equation}\label{e12}
		\left( \int_{\mathbb{R}^N} \left| f(x) \right|^{\frac{2(N-\gamma)}{N-2}} \rho(x) \, dx \right)^{\frac{N-2}{N-\gamma}} \leq C_S \, \int_{\mathbb{R}^N} \left| \nabla f(x) \right|^2 dx \qquad \forall f\in H^1(\mathbb{R}^N)
	\end{equation}
	holds for some positive constant $C_S>0$ depending on $ N,\gamma,K $; this is a matter of pure interpolation between the standard Sobolev inequality and Hardy's inequality. By applying \eqref{e12} to $ f = (\phi_R u )^{(q+m)/2} $ and plugging it in \eqref{m20}, we end up with
	\begin{equation}\label{m21}
		\left(\int_{\mathbb{R}^N} \left( \phi_R u  \right)^{s(q+1) +b } \, \rho \, dx \right)^{\frac 1 s} \leq  C \, q \left(\Lambda + \left\| u \right\|_{\infty, r}^{m-1} \right) \int_{\mathbb{R}^N} \left(\phi_R u \right)^{q+1} \rho \, dx \, ,
	\end{equation}
	where
	\begin{equation}\label{m21-bis}
		s:= \frac{N-\gamma}{N-2} > 1 \, , \qquad b:=  s(m-1)>0 \, . 
	\end{equation}
	The idea is now to carefully iterate \eqref{m21}. To this aim, we fix any $ p_0 \ge m+1 $, set $ \vartheta_0 := \frac{2-\gamma}{m-1} + \frac{N-\gamma}{p_0} $ and define recursively the following sequences (let $ j \in \mathbb{N} $): 
	\begin{equation}\label{m22}
		\begin{cases}
			p_{j+1}  = s p_j + b \, ,  \\
 \vartheta_{j+1}  = s \, \frac{p_j}{p_{j+1}}  \, \vartheta_j \, ,  
		\end{cases}
	\end{equation}
along with the quantity	
	\begin{equation*}\label{m22-bis}
	a_j:= \frac{1}{R^{\vartheta_j p_j}} \, \int_{\mathbb{R}^N} \left(\phi_R u \right)^{p_j} \rho \, dx \, .
	\end{equation*}
Raising \eqref{m21} to the $s$-th power, with $q=p_j -1$, and multiplying both sides by $R^{-\vartheta_{j+1}p_{j+1}}$, we obtain:
	\begin{equation}\label{m23} 
		a_{j+1} \leq C^s \, p_j^s \left( \Lambda + \left\| u \right\|_{\infty, r}^{m-1} \right)^s a_j^s \, .
	\end{equation}
Upon iterating \eqref{m23} $ j $ times, we deduce that
	\begin{equation}\label{m24}
		a_{j+1}^{\frac 1{p_{j+1}}} \leq \left[C \left(\Lambda + \left\|u \right\|^{m-1}_{\infty,r}\right) \right]^{\alpha_j} M_j \, a_0^{\beta_j} \, ,
	\end{equation}
	where
	\begin{equation}\label{m24b}
		\begin{cases}
			\alpha_j := \frac{s+s^2+\ldots+s^{j+1}}{p_{j+1}} \, , & \\
			M_j := \left(p_j^s \, p_{j-1}^{s^2} \ldots p_0^{s^{j+1}}\right)^{\frac 1{p_{j+1}}}  , & \\
			\beta_j := \frac{s^{j+1}}{p_{j+1}} \, ,
		\end{cases}
	\end{equation}
	and in view of \eqref{m22} we have
	\begin{equation}\label{m25}
		p_{j+1}= sp_j + b = \left(p_0 + \frac b{s-1}\right)s^{j+1} -\frac b{s-1}\,.
	\end{equation}
From \eqref{m21-bis}--\eqref{m22} and \eqref{m24b}--\eqref{m25}, it is straightforward to verify that 
	\begin{equation}\label{m26}
		\lim_{j\to \infty} \alpha_j = \lambda_{p_0} \, , \quad \lim_{j\to \infty} \beta_j = \frac{2-\gamma}{N-\gamma} \, \lambda_{p_0} \, , \quad \lim_{j\to \infty} \vartheta_j = \frac{2-\gamma}{m-1} \, , \quad \limsup_{j\to \infty} M_j =: \overline{M} \in (0,+\infty) \, , 
	\end{equation}
	where
	\begin{equation*}\label{def-lambda}
		\lambda_{p_0} := \frac{N-\gamma}{(N-\gamma)(m-1)+(2-\gamma)p_0}
	\end{equation*} 
	and $ \overline{M} $ only depends on $ N, m, \gamma, p_0 $. By virtue of \eqref{m26} we can let $ j \to \infty $ in \eqref{m24} to obtain 
	\begin{equation*}\label{m27}
		\frac{\left\| \phi_R u \right\|_{L^\infty \left(\mathbb{R}^N\right)}}{R^{\frac{2-\gamma}{m-1}}} \leq C \left(\Lambda + \left\| u \right\|_{\infty, r}^{m-1}\right)^{\lambda_{p_0}} \left( R^{-\frac{2-\gamma}{m-1} p_0 - N + \gamma} \int_{\mathbb{R}^N} (\phi_R u)^{p_0} \, \rho \, dx \right)^{\frac{2-\gamma}{N-\gamma} \, \lambda_{p_0}} ,
	\end{equation*}  
where we admit that $ C $ also depends on $ p_0 $. Therefore, if we take the supremum of both sides over $ R \ge r $ we end up with
	\begin{equation}\label{m27-bis}
		\left\| u \right\|_{\infty, r} \leq C \left(\Lambda + \left\| u \right\|_{\infty, r}^{m-1}\right)^{\lambda_{p_0}} \left\| u \right\|_{p_0 , r}^{\frac{2-\gamma}{N-\gamma} \, \lambda_{p_0} \, p_0} ,
	\end{equation}  
recalling the equivalence of the norms $ |\cdot|_{p,r} $ and $ \| \cdot \|_{p,r} $ (see \eqref{m7-bis}). Letting
	$$ 
	A := \left\| u \right\|^{m-1}_{\infty, r} ,
	$$
	estimate \eqref{m27-bis} reads 
	\begin{equation}\label{m28}
		A^{\frac 1{m-1}} \leq C \left(\Lambda + A\right)^{\lambda_{p_0}} \left\| u \right\|_{p_0 , r}^{\frac{2-\gamma}{N-\gamma} \, \lambda_{p_0} \, p_0} .
	\end{equation}
	A simple interpolation estimate yields 
	$$
	\left\| u \right\|_{p_0,r}^{p_0} \le A^{\frac{p_0-1}{m-1}} \, B \, , \qquad B := \left\| u \right\|_{1,r} ,
	$$
so that \eqref{m28}	entails
$$
A^{\frac 1{m-1}} \leq C \, B^{\frac{2-\gamma}{N-\gamma} \, \lambda_{p_0}} \left(\Lambda + A\right)^{\lambda_{p_0}} A^{\frac{2-\gamma}{N-\gamma} \, \frac{p_0-1}{m-1} \, \lambda_{p_0}} \, ,
$$
which is equivalent to 
\begin{equation*}\label{m28-bis}
A^{ \frac{N-\gamma}{2-\gamma} + \frac{1}{m-1} } \leq C^{\frac{N-\gamma}{(2-\gamma)\lambda_{p_0}}} \, B \left( \Lambda + A \right)^{\frac{N-\gamma}{2-\gamma}} .
\end{equation*}
We are therefore in position to apply the numerical result provided in \cite[Lemma 1.3]{BCP}, which ensures the existence of another constant $ C>0 $ depending only on $N,m,\gamma,k,K,p_0$ such that
\begin{equation}\label{m28-ter}
A \leq C  \left( \Lambda^{\frac{(N-\gamma)(m-1)}{(N-\gamma)(m-1)+2-\gamma}}   \, B^{\frac{(2-\gamma)(m-1)}{(N-\gamma)(m-1)+2-\gamma}}  + B^{m-1} \right)  .
\end{equation}
By undoing  $ A , B$ and recalling \eqref{m12a}--\eqref{m12b}, we realize that \eqref{m28-ter} is precisely \eqref{m11}. 
\end{proof}

\subsection{Approximate problems and \emph{a priori} estimates}\label{exiapp}

In this subsection we investigate in detail the main properties of weak energy solutions, and establish the crucial global estimates that we will need to prove existence. 

\begin{pro}\label{pro1}
	Let $u_0\in L^1\!\left(\mathbb{R}^N,\rho \right) \cap L^\infty\!\left(\mathbb{R}^N\right)$. Then there exists a unique weak energy solution $u$ of problem \eqref{wpme}, in the sense of Definition \ref{def2}. Moreover, we have that
	\begin{equation}\label{m37a}
	\left\| u(t) \right\|_{L^\infty\left(\mathbb{R}^N\right)} \leq \left\| u_0 \right\|_{L^\infty\left(\mathbb{R^N}\right)} \qquad \forall t\geq 0 \, ,
	\end{equation}
	and if $ v $ is the weak energy solution corresponding to another initial datum $v_0\in L^1\!\left(\mathbb{R}^N,\rho \right) \cap L^\infty\!\left(\mathbb{R}^N\right)$ then the following $ L^1 $-contraction estimate holds: 
	\begin{equation}\label{m37}
	\left\| u(t) - v(t) \right\|_{L^1\left(\mathbb{R}^N,\rho\right)} \leq \left\|u_0 - v_0 \right\|_{L^1\left(\mathbb{R}^N,\rho\right)}  \qquad \forall t\geq 0 \, .
	\end{equation}
	If in addition $ v_0 \le u_0 $ also the corresponding solutions are ordered, that is $v \le u$. In particular, $ u_0 \ge 0 $ implies $ u \ge 0 $, and  in such case the  inequality
	\begin{equation}\label{weak BC estimate}
	 \rho u_t \geq - \frac{\rho u}{(m-1) t} \qquad \text{in   } \mathcal{D}' \!\left(\mathbb{R}^N \times (0,+\infty) \right) 
	 	\end{equation}
	holds. Finally, solutions are strong, in the sense that 
		\begin{equation}\label{weak-strong}
	u_t \in L^\infty_{\mathrm{loc}}\! \left( (0,+\infty) ; L^1\!\left(\mathbb{R}^N,\rho\right) \right) . 
	\end{equation}
\end{pro}
\begin{proof}
	Under the running assumptions on $ \rho $ and $ u_0 $, the construction (and uniqueness) of a weak energy solution satisfying \eqref{m37a}--\eqref{m37} is by now rather standard and can be obtained by means of various approximation methods: we refer e.g.~to \cite[Theorem 3.1]{RV1} or to \cite[Proposition 6 and Theorem 3.4]{GMPo}, see also \cite[Theorems 2.1 and 2.2]{DQRV} for a fractional (unweighted) porous medium problem and \cite[Theorems 4.4 and 4.5]{GMP-dcds15} in the weighted case. The fact that in such works the weight is locally at least continuous is not an issue, as \eqref{weight-cond} guarantees that $ \rho $ can in turn be properly approximated by smooth functions. 
	
	As concerns the crucial estimate \eqref{weak BC estimate} for nonnegative solutions, which is originally due to B\'enilan and Crandall \cite{BC}, as observed in \cite[Theorem 3.7]{KRV} it is a matter of a formal time-scaling argument that can be easily proved for approximate solutions. The passage to the limit is then straightforward since the inequality is understood in the distributional sense. We refer to \cite[Lemma 8.1]{Vaz07} for the details.
	
The most delicate property to prove is \eqref{weak-strong}. To this aim, one can argue similarly to \cite[Theorem 8.2 and Corollary 8.3]{DQRV}. More precisely, thanks to \eqref{m37}, using the same time-scaling trick as in \cite[Lemma 8.5]{Vaz07} we have
\begin{equation}\label{strong-1}
\left\| u(t+h) - u(t) \right\|_{L^1\left(\mathbb{R}^N,\rho\right)} \le \frac{2 \left\|u_0  \right\|_{L^1\left(\mathbb{R}^N,\rho\right)}  }{(m-1)t} \left( h + o(h) \right)
\end{equation}
for all $ t,h>0 $, where $ o(h) $ is independent of $ t $ provided the latter ranges in compact subsets of $ (0,+\infty) $. Estimate \eqref{strong-1} actually shows that $ t \mapsto u(t) $ is locally Lipschitz in $ L^1\!\left( \mathbb{R}^N , \rho \right) $, but since the latter is not a reflexive space, this is not enough to claim that it is almost everywhere differentiable (see e.g.~\cite{Mik}). On the other hand, the fact that $u$ is a weak and bounded energy solution entails \cite[Lemma 3.3]{GMPo}
\begin{equation*}\label{strong-2}
\left( u^{\frac{m+1}{2}} \right)_t \in L^2_{\mathrm{loc}}\!\left( (0,+\infty) ; L^2\!\left( \mathbb{R}^N , \rho \right) \right) ,
\end{equation*}
so that in particular 
\begin{equation}\label{strong-3}
u^{\frac{m+1}{2}}  \in W^{1,1}_{\mathrm{loc}}\!\left( (0,+\infty) ; L^1_{\mathrm{loc}}\!\left( \mathbb{R}^N , \rho \right) \right) .
\end{equation}
In view of \eqref{strong-1} and \eqref{strong-3}, we are in position to apply \cite[Theorem 1.1 and Lemma 2.1]{BGa}, which ensures that $ u \in  W^{1,1}_{\mathrm{loc}}\!\left( (0,+\infty) ; L^1_{\mathrm{loc}}\!\left( \mathbb{R}^N , \rho \right) \right) $, but since $ u \in L^\infty\!\left((0,+\infty); L^1\!\left( \mathbb{R}^N , \rho \right)\right) $ upon letting $ h \to 0 $ in \eqref{strong-1} we end up with \eqref{weak-strong}.
\end{proof}

In the following two propositions we show that if $u_0\in L^1\!\left(\mathbb{R}^N,\rho\right)\cap L^\infty\!\left(\mathbb{R}^N\right)$ then estimates \eqref{smoothing estimate}, \eqref{key estimate}, \eqref{dependence on data 1} and \eqref{dependence on data 2} hold on the corresponding solution $u$ of \eqref{wpme}. These results are the weighted analogues of \cite[Lemma 1.4 and subsequent proofs]{BCP}. To this purpose, a crucial role is played by \eqref{m11}. In Subsection \ref{proof-exi} we will then use such estimates to prove general existence, and extend their validity to solutions whose initial data belong to $X$ by truncating $ u_0 $ and passing to the limit in the corresponding approximate solutions. 

\begin{pro}\label{l4} 
		Let $u_0\in L^1\!\left(\mathbb{R}^N,\rho \right) \cap L^\infty\!\left(\mathbb{R}^N\right)$ and $ u $ be the weak energy solution of problem \eqref{wpme}, in the sense of Definition \ref{def2}. Then there exist  positive constants $c_1 , c_2, c_3 $, depending only on $N,m,\gamma,k,K$, such that for all $ r \ge 1 $ we have
	\begin{equation}\label{contr-L1}
		\left\| u(t) \right\|_{1,r} \le c_2 \left\| u_0 \right\|_{1,r} \qquad \forall t \in \left( 0 , \tfrac{c_1}{ \left\| u_0 \right\|_{1,r}^{m-1}}  \right)  
	\end{equation}
	and the smoothing estimate 
	\begin{equation}\label{smooth-parab}
		\left\| u(t) \right\|_{\infty, r}  \leq c_3 \,  t^{-\lambda_1} \left\| u_0 \right\|_{1,r}^{\theta \lambda_1} \qquad \forall t \in \left( 0 , \tfrac{c_1}{ \left\| u_0 \right\|_{1,r}^{m-1}}  \right)  
	\end{equation}
	holds.
\end{pro}
\begin{proof}
	With no loss of generality we can suppose that $u_0$, and therefore $ u  $, are nonnegative. In fact, thanks to the ordering principle of Proposition \ref{pro1}, we infer that $ v \leq u \le w $, where $v $ is the (nonpositive) solution of problem \eqref{wpme} with $u_0$ replaced by $-u_0^-$ and  $w $ is the (nonnegative) solution of problem \eqref{wpme} with $u_0$ replaced by $u_0^+$. Hence
	\[ 
	 \left| u \right| \leq \left| v \right| \vee w \, , \quad  \left| v(0) \right| \vee w(0) = \left| u_0 \right| ,
	\]
	so that the validity of \eqref{smooth-parab} for $ v $ and $ w $ implies the validity of the same estimate for $u$. On the other hand $ v $ coincides with (minus) the nonnegative solution to \eqref{wpme} with initial datum $ u_0^- $, thus we can indeed restrict ourselves to nonnegative solutions. 
	
	Let therefore $u$ be a weak energy solution of problem \eqref{wpme}, with $ u \ge 0 $. If $ \phi_R $ is a smooth cut-off function as in \eqref{m1}--\eqref{m52}, for every $ R \ge 1 $ we have:
	\begin{equation}\label{m32}
		\begin{aligned}
			\frac{d}{d t} \int_{\mathbb{R}^N} u(t) \, \phi_R \, \rho\, dx  = & \, \int_{\mathbb{R}^N} u_t(t) \, \phi_R \, \rho \, dx  = \int_{\mathbb{R}^N} \left[ \Delta u(t)^m \right]  \phi_R \, dx = \int_{\mathbb{R}^N} u(t)^m \,  \Delta \phi_R \, dx \\
			\leq & \, \frac{C}{R^2} \, \int_{B_{2R} \setminus B_R} u(t)^{m}\, dx \leq \frac{C}{R^{2-\gamma}} \, \int_{B_{2R}  \setminus B_R} u(t)^{m} \, \rho \, dx \, , \\
		\end{aligned}
	\end{equation}
	where we let $C $ denote a general positive constant depending only on $ N,m, \gamma, k, K$, whose explicit value may differ from line to line. Note that in the last passage we have exploited \eqref{weight-cond}, similarly to \eqref{m19-bis}. We point out that the computation is justified for a.e.~$ t>0 $ thanks to \eqref{weak-strong}. By integrating \eqref{m32}  in time, exploiting the  $ L^1 $-continuity of $ u$, we obtain: 
	\begin{equation}\label{m33} 
		\begin{aligned}
			\int_{\mathbb{R}^N} u(t) \, \phi_R \, \rho \, dx \leq & \, \int_{\mathbb{R}^N} u_0 \, \phi_R \, \rho \, dx +  \frac{C}{R^{2-\gamma}} \, \int_0^t \int_{B_{2R}} u(s)^{m} \, \rho \, dx ds \\
			 \leq & \, \int_{B_{2R}} u_0 \, \rho \, dx + C \,  \int_0^t \frac{\left\| u(s) \right\|_{L^\infty(B_{2R})}^{m-1}}{(2R)^{2-\gamma}} \left( \int_{B_{2R}} u(s) \, \rho \, dx \right) ds \, ,
		\end{aligned}
	\end{equation}
	for all $ t>0 $. If we multiply both sides of \eqref{m33} by $R^{-\frac{2-\gamma}{m-1} -N+\gamma }$ and take the supremum over $R\geq r$, we deduce that the function $g(t) := \left\| u(t) \right\|_{1, r} $ satisfies the integral inequality 
	\begin{equation}\label{m34}
		g(t) \leq C \left( g(0) + \int_0^t \left\|u(s)\right\|_{\infty, r}^{m-1} g(s) \, ds \right) \qquad \forall t >0 \, .
	\end{equation}
	Let us observe that $ g $ is continuous since $ u(t) $ is a continuous function with values in $ L^1\!\left( \mathbb{R}^N ,\rho \right) $, and the latter space is continuously embedded in $ X $, whereas from the weak$ ^\ast $ lower semicontinuity of the $ L^\infty $ norm it is not difficult to check that $ t \mapsto \left\| u(t) \right\|_{\infty,r} $ is lower semicontinuous, thus measurable. In view of \eqref{weak BC estimate}, along with \eqref{energy} and \eqref{weak-strong}, for a.e.~$ t>0 $ we have that $ \nabla u(t)^m \in L^2\!\left( \mathbb{R}^N \right) $, $ \Delta u(t)^m \in L^1 \! \left( \mathbb{R}^N \right) $ and 
		\begin{equation*}\label{m10-time} 
	\Delta u(t)^{m} \geq - \frac{1}{(m-1)t} \, \rho u(t)  \, .
	\end{equation*}
	As a result, Proposition \ref{pro2}  is applicable, yielding
	\begin{equation}\label{m40}
		\left\| u(t) \right\|_{\infty, r}^{m-1}  \leq C \left( t^{-\lambda_1(m-1)} \left\| u(t) \right\|_{1,r}^{ \theta \lambda_1(m-1)} + \left\| u(t) \right\|_{1, r}^{m-1}\right) , 
	\end{equation}
	which can be plugged in \eqref{m34} to obtain 
	\begin{equation}\label{m41}
		g(t) \leq C g(0) +C  \, \int_0^t \left( s^{-\lambda_1(m-1)} \, g(s)^{\theta \lambda_1(m-1) + 1 } + g(s)^{m} \right) ds \, .
	\end{equation}
	Let us now introduce the solution to the Cauchy problem 
	\begin{equation}\label{m41-bis}
		\begin{cases}
			h'(t) =C \left[ t^{-\lambda_1(m-1)} \, h(t)^{\theta \lambda_1(m-1) + 1 } + h(t)^{m} \right] , \\
			h(0) =C g(0) \, , 
		\end{cases}
	\end{equation}
	which is smooth for $ t>0 $ and continuous down to $ t=0 $, up to the time $ \tau $ at which it blows up. Since the nonlinearities involved in the differential equation are locally Lipschitz and increasing, from \eqref{m41}--\eqref{m41-bis} it is not difficult to deduce a comparison principle which entails
	\begin{equation}\label{e3z}
	g(t) \leq h(t) \qquad \forall t \in (0,\tau) \, .
	\end{equation}
If $t$ is so small that 
	\begin{equation}\label{e5z}
		 t \leq \frac{1}{h(t)^{m-1}} \, ,
	\end{equation}
	then
	\[  
	h(t)^{m} = h(t)^{\lambda_1(m-1)^2} \, h(t)^{\theta \lambda_1(m-1)+1 }  \leq t^{-\lambda_1(m-1)} \, h(t)^{\theta \lambda_1(m-1)+1 } \,. 
	\]
	Hence, still by comparison, under \eqref{e5z} we have that  
	\begin{equation}\label{e6z}
		h(t) \leq H(t) \, ,
	\end{equation}
	where $H(t)$ is the solution to 
	\begin{equation}\label{e7z}
		\begin{cases}
			H'(t) = 2C \, t^{-\lambda_1(m-1)} \, H(t)^{\theta \lambda_1(m-1) + 1 } \, , \\
			H(0) =C g(0) \, .
		\end{cases}
	\end{equation}
	In fact \eqref{e7z} has an explicit solution, defined up to a finite blow-up time: 
	\begin{equation}\label{e8z}
		H(t)= \left[ \left( C g(0) \right)^{-\theta \lambda_1(m-1) } - 2C(m-1) \,  t^{\theta \lambda_1} \right]^{-\frac{1}{\theta \lambda_1 (m-1)}} .
	\end{equation}
	From \eqref{e3z}, \eqref{e5z} and \eqref{e6z}, we can therefore infer that 
	\begin{equation}\label{e9z}
		g(t) \leq H(t)
	\end{equation}
	provided
	\begin{equation}\label{e10z}
	 t \leq \frac{1}{H(t)^{m-1}} \, .
	\end{equation}
On the other hand, in view of \eqref{e8z}, condition \eqref{e10z} is equivalent to
	\begin{equation}\label{e11z}
		t \le \frac{c_1}{g(0)^{m-1}} \, ,
	\end{equation}
	for a suitable positive constant $ c_1 $ as in the statement. 
	
	Estimate \eqref{contr-L1} is therefore a consequence of \eqref{e8z}, \eqref{e9z} and \eqref{e11z}, while estimate \eqref{smooth-parab} just follows by combining \eqref{contr-L1} and \eqref{m40}.
	%
\end{proof}

\begin{pro}\label{continuous dependence estimates}
		Let $u_0,v_0\in L^1\!\left(\mathbb{R}^N,\rho \right) \cap L^\infty\!\left(\mathbb{R}^N\right)$ and $ u,v $ be the corresponding weak energy solutions of problem \eqref{wpme}, in the sense of Definition \ref{def2}. Let $ \alpha $ satisfy \eqref{cond-alpha} and $ r \ge 1 $. Then there exist positive constants $ c_4$ and $c_5 $, with $ c_4$ depending only on $ N,m,\gamma,k,K,r,\alpha,\| u_0 \|_{1,r} , \| v_0 \|_{1,r} $ and $ c_5 $ depending only on $ N,m,\gamma,k,K,\| u_0 \|_{1,r} , \| v_0 \|_{1,r} $, such that
	\begin{equation}\label{weighted continuous dependence}
	\left\| u(t)-v(t) \right\|_{L^1(\Phi_\alpha)} \leq \exp\!\left({c_4 \, t^{\theta\lambda_1}}\right) \left\| u_0-v_0 \right\|_{L^1(\Phi_\alpha)} \qquad \forall t \in \left( 0 , T_r(u_0) \wedge T_r(v_0) \right)
	\end{equation}
	and
	\begin{equation}\label{L1 continuous dependence}
	\left| u(t)-v(t) \right|_{1,r} \leq \exp\!\left({c_5 \, t^{\theta \lambda_1}}\right) \left| u_0-v_0 \right|_{1,r} \qquad \forall t \in \left( 0 , T_r(u_0) \wedge T_r(v_0) \right) , 
	\end{equation}
where $ T_r(f):=c_1/\left\| f \right\|_{1,r}^{m-1} $, the constant $ c_1 $ being the same as in Proposition \ref{l4}. 
\end{pro}
\begin{proof}
First of all we observe that $ \Delta u(t)^m, \Delta v(t)^m \in L^1\!\left( \mathbb{R}^N \right) $ for almost every $ t>0 $, since solutions are strong. Hence Kato's inequality is applicable (see \cite{K}), yielding
	\begin{equation}\label{sign trick}
		\int_{\mathbb{R}^N} \phi_R \, \sign(u(t)-v(t)) \, \Delta\!\left[u(t)^m-v(t)^m\right] dx\leq \int_{\mathbb{R}^N} \left( \Delta\phi_R \right) \left|u(t)^m-v(t)^m\right| dx \, ,
	\end{equation} 
where $ \phi_R $ is a smooth cut-off function as in \eqref{m1}--\eqref{m52} (let $ R \ge r $). In order to prove \eqref{L1 continuous dependence}, using \eqref{sign trick}, \eqref{weight-cond}, \eqref{smooth-parab} and the differential equation itself we obtain: 
	\begin{equation}\label{sign trick-2}
	\begin{aligned}
	& \,	\frac{d}{dt}\int_{\mathbb{R}^N}\phi_R \left|u(t)-v(t)\right|\rho \, dx \\
	 = & \,\int_{\mathbb{R}^N} \phi_R \, \sign(u(t)-v(t)) \, \Delta\!\left[u(t)^m - v(t)^m \right] dx\\
	\leq & \, \int_{\mathbb{R}^N} \left(\Delta\phi_R\right) \left|u(t)^m-v(t)^m \right| dx \\
	 \leq & \, \frac{C}{R^2} \, \int_{B_{2R}\setminus B_R} \left(|u(t)|^{m-1}\vee |v(t)|^{m-1}\right) \left|u(t)-v(t)\right|dx\\
\leq & \, C \left( \frac{\left\| u(t) \right\|_{L^\infty(B_{2R})}^{m-1}}{R^2} \vee \frac{\left\| v(t) \right\|_{L^\infty(B_{2R})}^{m-1}}{R^2} \right) \int_{B_{2R}\setminus B_R} \left|u(t)-v(t)\right|dx\\
\leq & \, C \left( \frac{\left\| u(t) \right\|_{L^\infty(B_{2R})}^{m-1}}{(2R)^{2-\gamma}} \vee \frac{\left\| v(t) \right\|_{L^\infty(B_{2R})}^{m-1}}{(2R)^{2-\gamma}} \right) \int_{B_{2R}\setminus B_R} \left|u(t)-v(t)\right| \rho \, dx  \\ 
\leq & \,  C \left( \left\|u(t)\right\|_{\infty,r}^{m-1} \vee \left\|v(t)\right\|_{\infty,r}^{m-1} \right) \int_{B_{2R}} \left|u(t)-v(t)\right| \rho \, dx  \\ 
\leq & \, C \left(\left\| u_0 \right\|_{1,r}^{\theta \lambda_1(m-1)} \vee \left\| v_0 \right\|_{1,r}^{\theta\lambda_1(m-1)} \right) \frac{1}{t^{\lambda_1(m-1)}} \, \int_{B_{2R}} \left|u(t)-v(t)\right| \rho \, dx \, ,
\end{aligned}
\end{equation}
where $ t < T_r(u_0) \wedge T_r(v_0) $ and $C$ is still a positive constants that may change from line to line and only depends on $ N,m,\gamma,k,K $. Integrating in time and proceeding similarly to the proof of Proposition \ref{l4}, it is readily seen that the function $ g(t):=|u(t)-v(t)|_{1,r} $ satisfies the integral inequality 
	\begin{equation*}\label{m41-ter}
		g(t) \leq g(0) + C \left(\left\| u_0 \right\|_{1,r} \vee \left\| v_0 \right\|_{1,r} \right)^{\theta \lambda_1(m-1)}  \int_0^t s^{-\lambda_1(m-1)} \, g(s) \, ds \, ,
	\end{equation*}
for all $ t < T_r(u_0) \wedge T_r(v_0) $. To finally reach \eqref{L1 continuous dependence}, we compare $g(t)$ to the solution of the Cauchy problem 
	\begin{equation*}\label{ode4}
		\begin{cases}
			h'(t)=\kappa \, t^{-\lambda_1(m-1)} \, h(t)\, , \\
			h(0)=\left|u_0-v_0\right|_{1,r} ,
		\end{cases}
		\quad \kappa := C \left(\left\| u_0 \right\|_{1,r} \vee \left\| v_0 \right\|_{1,r} \right)^{\theta \lambda_1(m-1)}  , 
	\end{equation*}
	namely
	\begin{equation*}
		h(t)=\left|u_0-v_0\right|_{1,r} e^{\frac{\kappa}{1-\lambda_1(m-1)}\, t^{1-\lambda_1(m-1)}} \, ,
	\end{equation*}
	and note that $1-\lambda_1(m-1)=\theta\lambda_1$.
		
	We employ a similar strategy to prove (\ref{weighted continuous dependence}). The following inequality can be obtained exactly as above by (formally) replacing $\phi_R$ with the function $\Phi_\alpha$ defined in \eqref{def-phi-alpha}:
	\begin{equation}\label{midpoint-1}
		\frac{d}{dt} \int_{\mathbb{R}^N} \Phi_\alpha \left|u(t)-v(t)\right|\rho \, dx \leq\int_{\mathbb{R}^N} \left( \Delta\Phi_\alpha \right) \left|u(t)^m-v(t)^m\right| dx  \qquad \text{for a.e. } t > 0 \, .
	\end{equation}
Next, we observe that  $\Delta\Phi_\alpha$ can be computed explicitly: 
	\begin{equation}\label{Lap-Phi}
		\Delta\Phi_\alpha(x)=-2\alpha \, \frac{N+\left(N-2\alpha-2\right)|x|^2}{\left(1+|x|^2\right)^{\alpha+2}} \qquad \forall x \in \mathbb{R}^N \, ,
	\end{equation}
	whence 
	\begin{equation}\label{lap-phi}
		\left|\Delta\Phi_\alpha(x)\right| \leq C_{N,\alpha} \, \frac{\Phi_\alpha(x)}{1+|x|^2} \qquad \forall x \in \mathbb{R}^N \, ,
	\end{equation}
	where $ C_{N,\alpha}>0 $ depends only on $ \alpha $ and $ N $. In fact one can justify rigorously \eqref{midpoint-1} by means of a standard cut-off argument, having in mind these decay estimates and recalling that $ u(t),v(t) \in L^1\!\left( \mathbb{R}^N , \rho \right) \cap L^\infty\!\left( \mathbb{R}^N \right) $. Substituting \eqref{lap-phi} into \eqref{midpoint-1}, we get
	\begin{equation}\label{midpoint}
\frac{d}{dt} \int_{\mathbb{R}^N} \Phi_\alpha \left|u(t)-v(t)\right|\rho \, dx		
		\leq C_{N,\alpha} \, m \, \int_{\mathbb{R}^N}\frac{|u(t)|^{m-1}\vee |v(t)|^{m-1}}{1+|x|^2} \left|u(t)-v(t)\right|\Phi_\alpha \, dx \, ,
	\end{equation}
for almost every $ t>0 $. To continue with our analysis, we would like an inequality of the following form:
	\begin{equation}\label{eq1}
	\frac{\left|u(x,t)\right|^{m-1}}{1+|x|^2}\leq C_r \left\| u(t) \right\|_{\infty,r}^{m-1} \rho(x) \qquad \text{for a.e. } x \in \mathbb{R}^N \, ,
	\end{equation}
	for a suitable positive constant $ C_r $ to be determined. To this end, first we observe that \eqref{eq1} is implied by 
	\begin{equation}\label{eq2}
	\frac{\left|u(x,t)\right|^{m-1}}{\left(1+|x|\right)^{2-\gamma}}\leq 2 C_r \,  k \left\| u(t) \right\|_{\infty,r}^{m-1}  \qquad \text{for a.e. } x \in \mathbb{R}^N \, ,
	\end{equation}
 thanks to \eqref{weight-cond}. Moreover, for all $ r \ge 1 $ we have
 $$
 \frac{1}{\left(1+|x|\right)^{2-\gamma}}\le \frac{r^{2-\gamma}}{\left(r+|x|\right)^{2-\gamma}} \qquad \forall x \in \mathbb{R}^N \, .
 $$
Hence, we can bound the LHS of \eqref{eq2} for a.e.~$ x \in \mathbb{R}^N $ as follows: 
	\begin{equation*}\label{eq3}
	\begin{aligned}
	\frac{\left|u(x,t)\right|^{m-1}}{\left(1+|x|\right)^{2-\gamma}} \leq  r^{2-\gamma} \, \frac{\left|u(x,t)\right|^{m-1}}{\left(r+|x|\right)^{2-\gamma}} \le & \, r^{2-\gamma} \, \sup_{n \in \mathbb{N}} \left\| \frac{u(t)}{\left(r+|x|\right)^{\frac{2-\gamma}{m-1}}} \right\|_{L^\infty(B_{n+1} \setminus B_n)}^{m-1} \\
	 \le & \, r^{2-\gamma} \, \sup_{n \in \mathbb{N}} \, \frac{\left\| u(t) \right\|_{L^\infty(B_{r+n})}^{m-1}}{\left(r+n \right)^{2-\gamma}} \le r^{2-\gamma} \left\| u(t) \right\|_{\infty,r}^{m-1} ,
	\end{aligned}
	\end{equation*}
finally obtaining \eqref{eq2} and thus \eqref{eq1} with $ C_r = r^{2-\gamma}/(2k) $.  
Substituting \eqref{eq1} (and the same inequality for $v$) into \eqref{midpoint} and applying \eqref{smooth-parab}, we get
$$
	\begin{aligned}
	& \, \frac{d}{dt} \int_{\mathbb{R}^N} \Phi_\alpha \left|u(t)-v(t)\right|\rho \, dx \\
	\leq  & \,  C_r \, C_{N,\alpha} \, m \left( \left\| u(t) \right\|_{\infty,r}^{m-1} \vee \left\| v(t) \right\|_{\infty,r}^{m-1} \right) \int_{\mathbb{R}^N} \left|u(t)-v(t)\right|\Phi_\alpha \, \rho \, dx \\ 
\leq & \, C_r \, C_{N,\alpha} \, m \, c_3^{m-1} \left(\left\| u_0 \right\|_{1,r}^{\theta \lambda_1(m-1)} \vee \left\| v_0 \right\|_{1,r}^{\theta\lambda_1(m-1)} \right) \frac{1}{t^{\lambda_1(m-1)}}  \int_{\mathbb{R}^N} \left|u(t)-v(t)\right|\Phi_\alpha \, \rho \, dx \, ,
	\end{aligned}
$$
for a.e.~$ t \in (0,T_r(u_0)\wedge T_r(v_0)) $. Integrating such a differential inequality we reach \eqref{weighted continuous dependence}. However, unlike \eqref{L1 continuous dependence}, we point out the additional dependence of the constant $c_4$ in \eqref{weighted continuous dependence} on $r$ and $\alpha$ through $ C_{N,\alpha} $ and $ C_r $.
\end{proof}

\subsection{Proof of existence} \label{proof-exi}
By using the \emph{a priori} estimates shown above, we are now in position to prove existence for general initial data in $X$. At the level of notation, we point out that the multiplying constants appearing in the statement of Theorem \ref{Existence} are the same as those appearing in Propositions \ref{l4} and \ref{continuous dependence estimates}, up to uppercase letters. 

\begin{proof}[Proof of Theorem \ref{Existence}]
First of all, we approximate the initial datum through the following  sequence of truncations: 
\begin{equation*}\label{truncations}
u_{0n} := \tau_n\!\left( u_0 \right) \chi_{B_n} \, , 
\end{equation*}
where
\begin{equation}\label{truncation-function}
\tau_n(s) :=
\begin{cases}
	s & \text{if}\ -n < s < n \, ,\\
	n & \text{if}\ n\leq s \, , \\
	-n &  \text{if}\ s \leq-n \, ,
\end{cases}
\end{equation}
for all $ n \in \mathbb{N} $. By construction, it is clear that $ \{ u_{0n} \} \subset L^1\!\left(\mathbb{R}^N,\rho\right) \cap L^\infty\!\left(\mathbb{R}^N\right) $, $ \left| u_{0n} \right| \le \left| u_0 \right| $ and $ u_{0n} \to u_0 $ pointwise as $ n \to \infty $. Hence, by Proposition \ref{technical}(2) we have that
\begin{equation}\label{truncations norm converges}
 \left\| u_{0n} \right\|_{1,r} \le  \left\| u_{0} \right\|_{1,r} \quad \forall n \in \mathbb{N} \, , \qquad \lim_{n\to \infty} \left\| u_{0n} \right\|_{1,r} =  \left\| u_{0} \right\|_{1,r}
\end{equation}
and, using dominated convergence as well, 
\begin{equation}\label{truncations converge}
\lim_{n \to \infty}	\left\| u_{0n}-u_0 \right\|_{L^1\left(\Phi_\alpha\right)} = 0 
\end{equation}
provided $ \alpha $ complies with \eqref{cond-alpha}. Let us consider the corresponding sequence $ \{ u_n \}$ of weak energy solutions of problems \eqref{q10}, whose existence and main properties are guaranteed by Proposition \ref{pro1}. Since each $ u_n $ is continuous from $[0,+\infty)$ into $L^1\!\left(\mathbb{R}^N,\rho\right)$, it is also continuous into $L^1\!\left(\Phi_\alpha\right)$. Upon noticing that $ T_r(u_{0n}) \ge T_r(u_{0}) $ and that the constants $ c_4,c_5 $ in \eqref{weighted continuous dependence}--\eqref{L1 continuous dependence} are increasing w.r.t.~dependence on $ \| \cdot \|_{1,r} $, from such estimates and \eqref{truncations norm converges}--\eqref{truncations converge} we readily infer that $ \{ u_n \} $ is a Cauchy sequence in $C\!\left([0,S];L^1\!\left(\Phi_\alpha\right)\right)$ for all $ S<T_r(u_0) $, thus it converges in such a space to a limit function which we call $ u $. Obviously the limit does not depend on $ S $, so $ u \in C\!\left(\left[ 0 , T_r(u_0) \right);L^1\!\left(\Phi_\alpha\right)\right) $ and therefore $ u \in C\!\left(\left[ 0 , T_r(u_0) \right);L^1_{\mathrm{loc}}\!\left(\mathbb{R}^N , \rho \right)\right) $ with $ u(0)=u_0 $.   

We now must demonstrate that $u$ is indeed a solution of \eqref{wpme} and that moreover it satisfies all of the estimates and other statements of Theorem \ref{Existence}. By Proposition \ref{l4}, we know that 
	\begin{equation}\label{smoothing estimate proof}
		\left\| u_n(t) \right\|_{1,r} \leq c_2 \left\| u_{0n} \right\|_{1,r}  \qquad \forall t \in \left( 0 , T_r(u_0) \right) 
	\end{equation}
	and
	\begin{equation}\label{key estimate proof}
		\left\| u_n(t) \right\|_{\infty,r} \leq c_3 \, t^{-\lambda_1} \left\| u_{0n} \right\|_{1,r}^{ \theta \lambda_1} \qquad \forall t \in \left( 0 , T_r(u_0) \right) ,
	\end{equation}
for all $ n \in \mathbb{N} $. On the other hand, since for every $ t>0 $ the sequence $ \{ u_n(t) \} $ converges to $ u(t) $ in $ L^1_{\mathrm{loc}}\!\left( \mathbb{R}^N , \rho \right) $, in view of \eqref{truncations norm converges} and Proposition \ref{technical}(2) we can pass to the limit in \eqref{smoothing estimate proof}--\eqref{key estimate proof} to get \eqref{smoothing estimate}--\eqref{key estimate}. If $ v_n $ is the weak energy solution of problem \eqref{q10} with initial datum $ v_{0n} $, then Proposition \ref{continuous dependence estimates} entails 
	\begin{equation}\label{weighted continuous dependence proof}
	\left\| u_n(t)-v_n(t) \right\|_{L^1(\Phi_\alpha)} \leq \exp\!\left({c_4 \, t^{\theta\lambda_1}}\right) \left\| u_{0n}-v_{0n} \right\|_{L^1(\Phi_\alpha)} \qquad \forall t \in \left( 0 , T_r(u_0) \wedge T_r(v_0) \right)
	\end{equation}
	and
	\begin{equation}\label{L1 continuous dependence proof}
\left| u_n(t)-v_n(t) \right|_{1,r} \leq \exp\!\left({c_5 \, t^{\theta \lambda_1}}\right) \left| u_{0n}-v_{0n} \right|_{1,r} \qquad \forall t \in \left( 0 , T_r(u_0) \wedge T_r(v_0) \right) , 
	\end{equation}
	for all $ n \in \mathbb{N} $. Due to $ L^1\!\left( \Phi_\alpha \right) $ convergence we can easily pass to the limit in \eqref{weighted continuous dependence proof} to reach \eqref{dependence on data 1}, whereas \eqref{dependence on data 2} follows from \eqref{L1 continuous dependence proof} and again Proposition \ref{technical}(2) up to noticing that $ \left| u_{0n}-v_{0n} \right|_{1,r} \le \left| u_{0}-v_{0} \right|_{1,r} $, consequence of the $1$-Lipschitz property of $ \tau_n(\cdot) $. If moreover $ u_{0} \le v_0 $, then $ u_{0n} \le v_{0n} $ by the monotonicity of $ \tau_n(\cdot) $, so that $ u_n \le v_n $ thanks to the ordering principle of Proposition \ref{pro1} and thus $ u \le v $ in $ \mathbb{R}^N \times \left( 0 , T_r(u_0) \wedge T_r(v_0) \right) $ upon letting $ n \to \infty $. 
 

In order to show that $u$ satisfies the very weak formulation \eqref{q50}, it is enough to note that \eqref{key estimate proof} implies the boundedness of $ \{ u_n \}  $ in $ L^\infty_{\mathrm{loc}}\!\left(\mathbb{R}^N\times \left(0,T_r(u_0)\right)\right) $, so we can deduce that $ u^m \in L^1_{\mathrm{loc}}\!\left( \mathbb{R}^N \times \left(0,T_r(u_0)\right) \right) $ and safely pass to the limit in the $n$-version of \eqref{q50}. The fact that, actually, $ u $ is a solution of problem \eqref{wpme} up to $ t=T(u_0) $ is just a consequence of the identity $ T(u_0) = \lim_{r \to +\infty} T_r(u_0) $, the sequence $ \{ u_n \} $ being independent of $ r $.

Let us finally deal with an initial datum $ u_0 \in X_0 $. In this case $  \lim_{r \to +\infty} \left\| u_0 \right\|_{1,r} = 0 $, hence we have that $ T(u_0) = \lim_{r \to +\infty} T_r(u_0) = +\infty $ and therefore the constructed solution $ u $ is global in time. By taking the limit as $ r \to +\infty $ in \eqref{key estimate}, we obtain:
$$
\lim_{r \to +\infty} \left\| u(t) \right\|_{\infty,r} \le  c_3 \, t^{-\lambda_1} \lim_{r \to +\infty} \left\| u_{0} \right\|_{1,r}^{ \theta \lambda_1} = 0 \qquad \forall t > 0 \, ,
$$
whence \eqref{slow} thanks to Proposition \ref{limsup-norm}. Now, for each $ n \in \mathbb{N} $ we know that $ u_n$ is continuous from $ [0,+\infty) $ into $ L^1\!\left( \mathbb{R}^N , \rho \right) $, which is continuously embedded in $ X $, so $ u_n \in C([0,+\infty);X) $. On the other hand, by  \eqref{L1 continuous dependence} we infer that
\begin{equation*}
\left| u_n(t)-u_j(t) \right|_{1,r} \le \exp\!\left({c_5 \, t^{\theta \lambda_1}}\right) \left| u_{0n}-u_{0j} \right|_{1,r} \qquad \forall t>0 \, , \ \forall n,j \in \mathbb{N} \, .
\end{equation*}
Proposition \ref{closedness} ensures that the RHS vanishes as $ n,j \to \infty $, because $ u_{0n} \to u_0 $ in $X$. This means that $ \{ u_n \} $ is a Cauchy sequence in $ C([0,S];X) $, for every $S>0$, and thus converges to the limit solution $ u $ in such space as well. As a result, we can deduce that $ u \in C([0,+\infty);X) $.
\end{proof}

\section{Uniqueness}\label{uniqueness}

The strategy of proof of Theorem \ref{thmuniq} is based on the so-called ``duality method'', {which was first exploited in \cite[Proposition 2.1]{BCP} for the large-data problem of the Euclidean porous medium equation} and more recently adapted to the manifold and fractional settings in \cite{GMPjmpa} and \cite{GMPfrac}, respectively. It consists of choosing a suitable family of test functions in  the weak formulation of problem \eqref{wpme} (recall Definition \ref{defsol}), which are not explicit but can be provided by carefully solving backward dual problems amounting to a weighted heat equation, see \eqref{formal-dual}.

The main issue here is to handle the possible lack of regularity of the weight close to $ x=0 $ (think of the limit case $ \rho(x) = |x|^{-\gamma} $). This is tackled in the proof of Proposition \ref{prop-uniq}. The latter establishes uniqueness under the additional assumption $ u_0 \in  X_\infty $ and an extra pointwise bound for the solutions, to be removed in the end, which amounts to requiring that the norm $ \| \cdot \|_{\infty,r} $ is stable along the evolution. Such a bound is crucial in order to successfully exploit the duality method. In a subsequent lemma we also show that the solutions constructed in Theorem \ref{Existence} do enjoy this property. 

\begin{pro}\label{prop-uniq}
Let $ u $ and $ v $ be any two solutions of problem \eqref{wpme}, corresponding to the same initial datum $u_0\in X_\infty$, such that
	\begin{equation}\label{e20z}
		\left|u(x,t)\right| \vee \left|v(x,t)\right| \leq C \left(1 + |x| \right)^{\frac{2-\gamma}{m-1}} \qquad \text{for a.e. } (x,t) \in \R^N \times (0, T) \, ,
	\end{equation} for some $ C>0 $. Then $ u = v $.
\end{pro} 
\begin{proof}

We will proceed through several steps.
	
\medskip 	
	
\noindent \textsc{Step 1} (Rewriting the very weak formulation). 

Clearly, the thesis follows if we are able to show that $u(\tau)=v(\tau)$ for every $ \tau \in (0,T) $. In the sequel, for notational convenience, we will still use $ T $ instead of $ \tau $. For any test function $ \xi \in C^\infty_c\!\left( \R^N \times [0,T] \right) $, from \eqref{q50} and the fact that $ u,v $ are continuous from $ [0,T] $ into $ L^1_{\mathrm{loc}}\!\left(\R^N,\rho\right) $, and share the same initial datum, it is not difficult to deduce that the following identity holds:
	\begin{equation}\label{n3}
		\int_0^T \int_{\R^N} \left[ (u - v) \, \rho \,  \xi_t  + \left(u^m - v^m\right) \Delta \xi \right] dx dt = \int_{\R^N} \left[ u(T) - v(T) \right] \xi(T) \, \rho \, dx \, .
	\end{equation}
Letting  
	\begin{equation}\label{n4}
		a(x,t):= 
		\begin{cases}
			\frac{u^m(x,t) - v^m(x,t)}{u(x,t)- v(x,t)} & \text{if } u(x,t)\neq v(x,t) \, , \\
			0  & \text{if } u(x,t) = v(x,t) \, ,
		\end{cases}
	\end{equation}
	thanks to \eqref{e20z} and the monotonicity of $ s \mapsto s^m $ there exists another positive constant (which we call again $C$) such that
	\begin{equation}\label{n5}
		0 \leq  a(x,t) \leq  C \left(1 + |x|\right)^{2-\gamma} \qquad \text{for a.e. } (x,t) \in \R^N \times (0, T)\,.
	\end{equation}
	Hence, in view of \eqref{n4}, we can rewrite \eqref{n3} as
	\begin{equation}\label{n6}
		\int_0^T \int_{\R^N} (u - v) \left(\rho \, \xi_t + a \, \Delta \xi \right) dx dt  = \int_{\R^N} \left[ u(T) - v(T) \right] \xi(T) \, \rho \, dx \, .
	\end{equation}
	
	\medskip 
	
	\noindent \textsc{Step 2} (Introducing the approximate dual problem). 

	We choose an arbitrary nonnegative function $ \omega \in C^\infty_c\!\left(\R^N \setminus \{ 0 \} \right)  $, and fix $ R_0,\delta_0>0 $ in such a way that 
	\begin{equation}\label{support-omega}
	\operatorname{supp} \, \omega \subseteq \overline{B}_{R_0} \setminus B_{\delta_0} \, .
	\end{equation}
	Then we pick $ R>R_0 + 1 $ and $ 0 < \delta < \delta_0/2 $; in the sequel, $ R_0 $ and $ \delta_0  $  will not change, whereas $ R $ and $ \delta $ will eventually go to $ +\infty $ and $ 0 $, respectively. Ideally, we would like to substitute into \eqref{n6} the solution of the following backward ``dual'' problem: 
	\begin{equation}\label{formal-dual}
	\begin{cases}
	\rho \, \xi_t + a \, \Delta \xi = 0 & \text{in } \R^N \times (0,T) \, ,\\
	\xi =\omega  & \text{on } \R^N \times \{ T \} \, ,
	\end{cases}
	\end{equation}
	which would immediately yield $ \int_{\R^N} \left[ u(T) - v(T) \right] \omega \, \rho \, dx = 0 $ and thus $ u(T)=v(T) $ due to the arbitrariness of $ \omega $. However, there are two main issues: the functions $ a,\rho $ are in general not regular, and $ a $ might be degenerate whereas $ \rho $ is allowed to be singular at the origin. On top of that, $ \xi $ cannot be expected to have a compact support. In order to overcome them, we resort to suitable approximations. As for $ \rho $, we pick a sequence of smooth and positive functions  $ \{ \rho_n \}  $ such that 
	\begin{equation}\label{rho-n-1}
	\lim_{n \to \infty} \rho_n(x) = \rho(x) \qquad \text{for a.e. } x \in \R^N
	\end{equation}
	and 
	\begin{equation}\label{rho-n-2}
	\hat{k} \left( 1 + |x|^2  \right)^{-\frac \gamma 2}  \le \rho_n(x) \le \hat{K} \left| x \right|^{-\gamma} \qquad \forall x \in \mathbb{R}^N \setminus \{ 0 \} \, , \ \forall n \in \mathbb{N} \, ,
	\end{equation}
	where $ \hat{k},\hat{K} $ are positive constants depending only on $ \gamma,k,K $. The existence of such a sequence is readily seen e.g.~by a standard mollification argument. We also approximate $ a $ by means of a sequence $ \{ a_n \}  \subset  C^\infty\!\left( \R^N \times [0, T] \right) $ such that 
	\begin{equation}\label{n20}
		a_n(x,t) \leq \hat{C}  \left(1+ |x|^2\right)^{\frac{2-\gamma}{2}} \qquad  \forall (x,t) \in \R^N \times [0, T] \, , \ \forall n \in \mathbb{N} \, ,
	\end{equation}
	for some positive constant $ \hat{C} $ depending only on $ C,\gamma $. The construction of $ \{ a_n \} $ will be specified in the last step, and we point out that it will also depend on  $R,\delta$, which at this stage are fixed. Now, for each $ n \in \mathbb{N} $, we can introduce the solution $\xi_n$ of the following backward parabolic problem:
	\begin{equation}\label{n9}
		\begin{cases}
			\rho_n \,  \partial_t \xi_n  + a_n \, \Delta \xi_n = 0 & \text{in } \left( B_R \setminus \overline{B}_{\delta} \right) \times (0, T) \, , \\ 
			\xi_n = 0 & \text{on } \left( \partial B_R \cup \partial B_\delta \right) \times (0,T) \, , \\
			\xi_n = \omega  & \text{on } \left( B_R \setminus \overline{B}_{\delta} \right) \times \{ T \} \, .
		\end{cases}
	\end{equation}
Note that, after applying a time change, \eqref{n9} is nothing other than a linear forward parabolic problem with smooth coefficients (in non-divergence form).  As the initial datum is regular and compactly supported in the domain $ B_R \setminus \overline{B}_{\delta} $ and homogeneous boundary conditions are prescribed, standard parabolic theory ensures that $ \xi_n $ is also smooth (see e.g.~\cite[Chapter V]{Lieb}). Hence, dividing the differential equation by $ \rho_n $, multiplying by $ \Delta \xi_n $ and integrating we end up with the crucial identity 
	\begin{equation}\label{eq-energy}
		\frac{1}{2} \int_{B_R \setminus B_\delta } \left| \nabla \xi_n(0) \right|^2 dx + \int_0^T \int_{B_R \setminus B_\delta } \left| \Delta \xi_n \right|^2 \frac{a_n}{\rho_n} \, dxdt = \frac{1}{2}\int_{B_R \setminus B_\delta} \left| \nabla \omega \right|^2 dx =  \frac 1 2 \left\| \nabla \omega \right\|_{L^2\left( \R^N \right)}^2  .
	\end{equation}
Moreover, $ \omega $ being nonnegative and bounded, by the maximum principle (see e.g.~\cite[Chapter II]{Lieb}) we have 
	\begin{equation*}\label{Lp-estimates-xi}
	0 \le	\xi_n  \le \left\| \omega \right\|_{L^\infty\left( \R^N \right)} \qquad \text{in } \left(  \overline{B}_R \setminus B_\delta \right) \times [0,T] \, .
	\end{equation*}
In particular, from the homogeneous Dirichlet boundary condition in \eqref{n9}, we can deduce that
	\begin{equation}\label{n43}
		\frac{\partial \xi_n}{\partial \nu} \leq 0 \qquad \text{on } \left(  \partial B_R \cup \partial B_\delta \right) \times (0, T) \, ,
	\end{equation}
	where $ \frac{\partial}{\partial \nu} $ stands for the outer normal derivative. 
	
	\medskip 
	
	\noindent \textsc{Step 3} (Dealing with remainder terms). 
	
	Now we aim at plugging $ \xi = \xi_n $ into \eqref{n6} and estimate the remainder terms due to the fact that $ a,\rho $ have been replaced by $ a_n,\rho_n$. However, a preliminary issue to tackle is the lack of regularity of $ \xi_n $ when one extends it to the whole $ \R^N $. To this end, for every small enough $ \varepsilon>0 $ we can select a smooth spacial cut-off function $ \varphi_\varepsilon  $ such that
	\begin{equation}\label{cutoff-eps}
	0 \leq \varphi_\varepsilon \leq 1 \quad \text{in } \R^N \, , \qquad \varphi_\varepsilon \equiv 1 \quad \text{in } \overline{B}_{R-2\varepsilon} \setminus B_{\delta+2\varepsilon}  \, , \qquad \varphi_\varepsilon \equiv 0 \quad \text{in }  B_{R-\varepsilon}^c \cup \overline{B}_{\delta+\varepsilon}  \, ,
	\end{equation}
	and 
		\begin{equation}\label{cutoff-eps-bis}
	\left\| \nabla\varphi_\varepsilon \right\|_{L^\infty \left( \R^N \right)} \leq \frac{c}{\varepsilon} \, ,  \qquad \left\| \Delta \varphi_\varepsilon \right\|_{L^\infty \left( \R^N \right)}  \leq \frac{c}{\varepsilon^2} \,  ,
	\end{equation}
	for some constant $c>0$ not depending on $\varepsilon$. The test function $ \xi = \varphi_\varepsilon \, \xi_n $ is admissible in \eqref{n6}, yielding 
	\begin{equation}\label{plugin}
	\int_0^T\int_{\mathbb{R}^N} (u-v)\left[\rho\, \varphi_\varepsilon \, \partial_t\xi_n + a \, \Delta\!\left(\varphi_\varepsilon \, \xi_n\right)\right] dxdt =\int_{\R^N} \left[ u(T) - v(T) \right]  \varphi_\varepsilon \, \omega \, \rho \, dx \, . 
	\end{equation}
	We develop the LHS of \eqref{plugin} by using the differential equation in \eqref{n9} and the product formula for the Laplacian:
	\begin{equation}\label{mainuniqest}
		\begin{aligned}
			\int_{\R^N} \left[ u(T) - v(T) \right]  \varphi_\varepsilon \, \omega \, \rho \, dx = & \, \underbrace{\int_0^T\int_{\mathbb{R}^N}(u-v) \left(a-a_n\right) \varphi_\varepsilon \,  \Delta\xi_n \, dxdt}_{I_{n,\varepsilon}} \\ 
			&+ \underbrace{\int_0^T\int_{\mathbb{R}^N} \left(u^m-v^m\right)\left(2 \left\langle \nabla \varphi_\varepsilon ,  \nabla\xi_n \right\rangle +\xi_n \, \Delta\varphi_\varepsilon\right)dxdt}_{J_{n,\varepsilon}} \\
			&+\underbrace{\int_0^T\int_{\mathbb{R}^N}(u-v) \left(\rho-\rho_n\right) \varphi_\varepsilon \, \partial_t \xi_n \, dxdt}_{E_{n,\varepsilon}} \, .
		\end{aligned}
	\end{equation}
	In order to estimate $J_{n,\varepsilon}$, using \eqref{cutoff-eps}--\eqref{cutoff-eps-bis} we obtain
	\begin{equation}\label{jne-estimate1}
		\left|J_{n,\varepsilon}\right|\leq 3 c \int_0^T\int_{ \left( B_{R}\setminus B_{R-2\varepsilon} \right) \cup \left( B_{\delta+2\varepsilon} \setminus B_\delta \right) } \left|u^m-v^m\right|\left(\frac{\left|\nabla\xi_n\right|}{\varepsilon}+\frac{\left|\xi_n\right|}{\varepsilon^2}\right)dxdt \, .
	\end{equation}
	Since $\xi_n \equiv 0$ on $\partial B_R \cup \partial B_\delta $ and $\xi_n$ is regular up to the boundary, we have
	$$
	 	\sup_{\left( B_{R}\setminus B_{R-2\varepsilon} \right) \times (0,T) } \left| \xi_n \right| \leq 2\varepsilon \sup_{\left( B_{R}\setminus B_{R-2\varepsilon} \right) \times (0,T) } \left| \nabla  \xi_n \right|
	$$
	and
		$$
	\sup_{\left( B_{\delta+2\varepsilon}\setminus B_{\delta} \right) \times (0,T) } \left| \xi_n \right| \leq 2\varepsilon \sup_{\left( B_{\delta+2\varepsilon}\setminus B_{\delta} \right) \times (0,T) } \left| \nabla  \xi_n \right| . 
	$$
	Thus, taking advantage of the above bounds and \eqref{e20z}, estimate \eqref{jne-estimate1} becomes
	\begin{equation}\label{jne-estimate2}
	\begin{aligned}
		\left|J_{n,\varepsilon}\right|\leq & \, \tilde{C} \, R^{\frac{(2-\gamma)m}{m-1}} \sup_{\left( B_{R}\setminus B_{R-2\varepsilon} \right) \times (0,T) } \left| \nabla  \xi_n \right| \frac{1}{\varepsilon} \int_0^T\int_{B_{R}\setminus B_{R-2\varepsilon}}dxdt   \\
		& + \tilde{C} \, \sup_{\left( B_{\delta+2\varepsilon}\setminus B_{\delta} \right) \times (0,T) } \left| \nabla  \xi_n \right| \frac{1}{\varepsilon} \int_0^T\int_{B_{\delta+2\varepsilon}\setminus B_{\delta}}dxdt  \, ,
		\end{aligned}
	\end{equation}
	for another positive constant $ \tilde{ C } $ depending only on $ m, C,c,\delta_0 $. Again, by the regularity of $\xi_n$ near the boundary, we infer that
	\begin{equation*}
		\lim_{\varepsilon \to 0 } \, \sup_{\left( B_{R}\setminus B_{R-2\varepsilon} \right) \times (0,T) } \left| \nabla  \xi_n \right|  =\sup_{\partial B_R \times (0,T) } \left|\frac{\partial\xi_n}{\partial\nu} \right| , \quad 	\lim_{\varepsilon \to 0 } \, \sup_{\left( B_{\delta+2\varepsilon}\setminus B_{\delta} \right) \times (0,T) } \left| \nabla  \xi_n \right|  =\sup_{\partial B_\delta \times (0,T) } \left|\frac{\partial\xi_n}{\partial\nu} \right| , 
	\end{equation*}
	so we can take limits as $ \varepsilon \to 0 $ in \eqref{jne-estimate2} to get 
	\begin{equation}\label{jn-estimate}
		J_n:=\limsup_{\varepsilon \to 0} \left| J_{n,\varepsilon} \right| \leq \tilde{ C } \left( R^{N-1+\frac{(2-\gamma)m}{m-1}} \sup_{\partial B_R \times (0,T) } \left|\frac{\partial\xi_n}{\partial\nu} \right| + \delta^{N-1} \sup_{\partial B_\delta \times (0,T) } \left|\frac{\partial\xi_n}{\partial\nu} \right| \right) , 
	\end{equation}
	where $ \tilde{ C } $ is a constant as above which in addition depends on $ N, T $. To conclude our estimate of $ J_n $ it is only left to obtain good bounds on the normal derivatives of $ \xi_n $, which will be done in the next step. As concerns $ I_{n,\varepsilon} $, H\"{o}lder's inequality yields  
	\begin{equation}\label{ine-estimate} 
	\begin{aligned}
	\left|	I_{n, \varepsilon} \right| \leq I_n :=  & \, \left(\int_0^T\int_{B_{R} \setminus B_\delta }(u-v)^2 \, \frac{(a-a_n)^2}{a_n} \, \rho_n \, dxdt\right)^{\frac 1 2}\left(\int_0^T\int_{B_{R} \setminus B_\delta } \left|\Delta\xi_n\right|^2   \frac{a_n}{\rho_n} \,  dxdt \right)^{\frac 1 2}\\
		\leq & \, C(R,\delta) \left(\int_0^T\int_{B_{R}\setminus B_\delta}\frac{(a-a_n)^2}{a_n} \, dxdt\right)^{\frac 1 2} ,
	\end{aligned}
	\end{equation}
	where we have used \eqref{e20z}, \eqref{rho-n-2} and \eqref{eq-energy} in the second passage, the positive constant $ C(R,\delta) $ depending on $ m,\gamma,C,\hat{K}, \omega,T,R,\delta $ but not on $n$. We will finish this estimate later by making more specific assumptions on the construction of $a_n$. Finally, we control the error term $E_{n,\varepsilon}$ as follows:
	\begin{equation*}\label{err-term}
	\begin{aligned}
	\left| E_{n,\varepsilon} \right| \leq & \, \int_0^T\int_{B_{R}\setminus B_{\delta}} \left|u-v\right|\left|\rho-\rho_n\right| \left|\partial_t\xi_n \right| dxdt\\
		 \leq & \, 2 C \left(1+R\right)^{\frac{2-\gamma}{m-1}}\left(\int_0^T\int_{B_{R}\setminus B_{\delta}} \left( \rho-\rho_n \right)^2 \frac{a_n}{\rho_n}  \, dxdt\right)^{\frac 1 2}\left(\int_0^T\int_{B_{R}\setminus B_{\delta}} \left|\Delta\xi_n\right|^2 \frac{a_n}{\rho_n} \, dxdt\right)^{\frac 1 2}\\
		\leq & \, C(R,\delta) \left(\int_0^T\int_{B_{R} \setminus B_{\delta}} \left(\rho-\rho_n\right)^2 dxdt\right)^{\frac 1 2} ,
	\end{aligned}
	\end{equation*}
	where we have used again \eqref{e20z}, H\"older's inequality, \eqref{rho-n-2}, \eqref{n20}, \eqref{eq-energy}, and the differential equation in \eqref{n9}. Here $ C(R,\delta) $ is a positive constant as above, depending only on the quantities $ m,\gamma,C,\hat{C},\hat{k}, \omega,T,R,\delta $ but not on $n$. Hence, by virtue of \eqref{rho-n-1}--\eqref{rho-n-2} and dominated convergence, we get that
	\begin{equation}\label{error result}
\lim_{n \to \infty}	 \limsup_{\varepsilon \to 0} \left| E_{n,\varepsilon} \right| =0 \, .
	\end{equation}
	
	\medskip 
	
	\noindent \textsc{Step 4} (Bounding normal derivatives). 
	
	In order to control $J_n$, from \eqref{jn-estimate} it is clear that we need to suitably estimate 
	\[ 
\sup_{\partial B_R\times(0, T)} \left|\frac{\partial \xi_n}{\partial \nu}\right| \qquad \text{and} \qquad \sup_{\partial B_\delta \times(0, T)} \left|\frac{\partial \xi_n}{\partial \nu}\right| .  \]
Let us start from the normal derivative on $ \partial B_R $. For any $ \sigma >0$ and $\beta > \gamma / 2 $ we define
	\begin{equation}\label{n21b}
		\eta(x,t):=\frac{\kappa \, e^{\sigma(T-t)}}{\left(1+|x|^2\right)^{\beta-\frac \gamma 2}} \qquad \forall (x,t)\in \mathbb{R}^N\times [0, T]\,,
	\end{equation}
	where we select $ \kappa>0 $ such that 
	\[ 
	\kappa \geq \| \omega \|_{L^\infty\left(\mathbb{R}^N\right)} \left(1+R_0^2\right)^{\beta-\frac \gamma 2}  .\]
	Thus, recalling the support properties of $ \omega $,
	\begin{equation}\label{n33}
		\eta \geq \left\| \omega \right\|_{L^\infty\left(\mathbb{R}^N\right)} \chi_{B_{R_0}}\geq \omega \qquad \text{in } \R^N \times [0,T] \,.
	\end{equation}
	We claim that, for a suitable choice of the parameters $\sigma$ and $\beta$, the function $\eta$ satisfies (for any fixed $ n \in \mathbb{N} $)
	\begin{equation}\label{n33b}
	\rho_n \, \eta_t + a_n \, \Delta \eta \leq 0 \qquad \text{in } \mathbb{R}^N\times (0, T) \, .
	\end{equation}
	First, an elementary computation shows that $\rho_n \eta_t = -\sigma\rho_n \eta$. Then we notice that, thanks to \eqref{rho-n-2}, \eqref{n20} and \eqref{Lap-Phi} with $ \alpha = \beta - \gamma/2 $, 
\begin{equation*}\label{n34b}
		\left| a_n \, \Delta\eta \right| \leq C' \,  \eta \left(1+|x|^2\right)^{-\frac \gamma 2} \le \frac{C'}{\hat{k}} \, \rho_n \, \eta \, ,
\end{equation*}
where $ C' $ only depends on $ N,\gamma,\hat{C},\beta$. Combining the above results, we infer that \eqref{n33b} holds as long as $ \beta>\gamma/2 $ and 
	\begin{equation*}\label{alphacond}
		\sigma \ge \frac{C'}{\hat{k}} \, .
	\end{equation*}
Thus, by \eqref{n33}--\eqref{n33b} and the maximum principle,
	\begin{equation*}
		\eta \geq \xi_n \qquad \text{in } \overline{B}_R \setminus B_\delta \times [0,T] \, . 
	\end{equation*}
We next observe that the time-independent function 
	\begin{equation}\label{barrier-2}
g(x) := \frac{d_1}{|x|^{N-2}} + d_2 \, ,
\end{equation}
	with the choices
	$$
	d_1 = \frac{R^{N-2}\,(R-1)^{N-2}}{R^{N-2}-(R-1)^{N-2}} \, \frac{\kappa \, e^{\sigma T}}{\left(1+(R-1)^2\right)^{\beta-\frac \gamma 2  }} \, , \qquad d_2 = - \frac{d_1}{R^{N-2}} \, ,
	$$
	satisfies $ \Delta g = 0 $ in $ \R^N \setminus \{ 0 \} $ and, by construction, 
	$$
	g = 0 = \xi_n \quad \text{on } \partial B_R \times (0,T) \, , \qquad g = \frac{\kappa \, e^{\sigma T}}{\left(1+(R-1)^2\right)^{\beta-\frac \gamma 2  }} \ge  \eta \ge \xi _n \quad \text{on } \partial B_{R-1} \times (0,T) \, , 
	$$
	whereas the inequality $ g \ge 0 = \xi_n  $ on $ \left( B_R \setminus \overline{B}_{R-1} \right) \times \{ T \} $ is trivial due to the support properties of $ \omega $. Hence, by the maximum principle applied in $ \left( B_R \setminus \overline{B}_{R-1} \right) \times (0,T) $, we deduce that 
		\begin{equation*}\label{supersol}
		g \geq \xi_n \qquad \text{in } \left( {B}_R \setminus \overline{B}_{R-1} \right) \times (0,T) \, . 
	\end{equation*}
Because both $ g $ and $ \xi_n $ vanish on $ \partial B_R $, then it is clear that 
	\begin{equation*}
		\frac{\partial}{\partial\nu}\left(g-\xi_n\right) \leq 0 \qquad\text{on } \partial B_R \times (0,T) \, . 
	\end{equation*}
Recalling \eqref{n43}, this entails 
	\begin{equation}\label{hhh}
		\sup_{\partial B_R\times(0, T)} \left|\frac{\partial \xi_n}{\partial \nu}\right|\leq \sup_{\partial B_R\times(0, T)} \left|\frac{\partial g}{\partial \nu}\right| .
	\end{equation}
	Therefore, we need to compute $ \frac {\partial g} {\partial\nu} $ on $ \partial B_R $:
	\begin{equation*}
	\left .	\frac{\partial g}{\partial \nu} \right |_{\partial B_R}  = - \frac{(N-2)\,d_1}{R^{N-1}} = -\frac{N-2}{R} \, \frac{(1-1/R)^{N-2}}{1-(1-1/R)^{N-2}} \, \frac{\kappa \, e^{\sigma T}}{\left(1+(R-1)^2\right)^{\beta-\frac \gamma 2  }}  \, . 
	\end{equation*}
	Hence 
	\begin{equation}\label{est-R}
		\sup_{\partial B_R\times(0, T)} \left|\frac{\partial g}{\partial \nu}\right| \leq \frac{c}{R^{2\beta-\gamma}}
	\end{equation}
	for some positive constant $ c $ depending only on $ N,\gamma,\beta,R_0,\kappa,\sigma,T $. Now we estimate the normal derivative of $ \xi_n $ on $ \partial B_\delta $. We use again a barrier of the form \eqref{barrier-2}, with the choices 
	$$
		d_1 = - \frac{\delta_0^{N-2} \, \delta^{N-2}}{\delta_0^{N-2}-\delta^{N-2}}  \left\| \omega \right\|_{L^\infty\left(\mathbb{R}^N\right)} , \qquad d_2 = - \frac{d_1}{\delta^{N-2}} \, .
	$$
This ensures that $ g $ is harmonic in $ B_{\delta_0} \setminus \overline{B}_{\delta} $ and satisfies 
		$$
	g = 0 = \xi_n \quad \text{on } \partial B_\delta \times (0,T) \, , \qquad g = \left\| \omega \right\|_{L^\infty\left(\mathbb{R}^N\right)} \ge \xi _n \quad \text{on } \partial B_{\delta_0} \times (0,T) \, , 
	$$
with in addition  $ g \ge 0 = \xi_n  $ on $ \left( B_{\delta_0} \setminus \overline{B}_{\delta} \right) \times \{ T \} $ thanks to \eqref{support-omega}. Thus, arguing as above, we can infer that 
	\begin{equation}\label{hhh-bis}
		\sup_{\partial B_\delta \times(0, T)} \left|\frac{\partial \xi_n}{\partial \nu}\right|\leq \sup_{\partial B_\delta \times(0, T)} \left|\frac{\partial g}{\partial \nu}\right| = \frac{N-2}{\delta} \, \frac{\delta_0^{N-2} \left\| \omega \right\|_{L^\infty\left(\mathbb{R}^N\right)} }{\delta_0^{N-2}-\delta^{N-2}} \, T \, \le \frac{c}{\delta} \, ,
	\end{equation}
for another $ c>0 $ depending only on $ N,\delta_0,\omega,T $, that we do not relabel. We finally estimate $ J_n $ by means of \eqref{jn-estimate}, \eqref{hhh}, \eqref{est-R} and \eqref{hhh-bis}, obtaining	
	\begin{equation}\label{jn-estimate-final}
		J_n \le  \tilde{ C } c \left( R^{N-1+\frac{(2-\gamma)m}{m-1}-2\beta + \gamma	}  + \delta^{N-2} \right) .
	\end{equation}
	
		\medskip 
	
	\noindent \textsc{Step 5} (Taking limits and conclusion). 
	
		From \eqref{jn-estimate-final}, we have 
	\begin{equation}\label{j_n control}
	\lim_{ \substack{R\to+\infty \\ \delta \to 0} } \limsup_{n \to \infty} J_n = 0 
	\end{equation}
	since $ N \ge 3 $, provided
	\begin{equation*}\label{betacond}
		\beta>\frac{N-1}{2}+\frac{(2-\gamma)m}{2(m-1)}+\frac{\gamma}{2} \, ,
	\end{equation*} 
	which we will therefore take to be an additional assumption in \eqref{n21b}. Hence, we are left with the analysis of the remainder $I_n$ as in \eqref{ine-estimate}; to this end, we need to properly construct the sequence $ \{ a_n \} $. For every fixed $ R $ and $ \delta $ as above, we choose a sequence $ \{ \alpha_n \} $ of nonnegative and smooth functions such that 
	\begin{equation}\label{j_n control-2}
	\int_0^T \int_{B_R \setminus B_\delta} \left( a - \alpha_n \right)^2  dx dt \leq \frac 1 {\left( n + 1 \right)^2} \qquad \forall n \in \mathbb{N}
	\end{equation} 
	and \eqref{n20} holds with $a_n$ replaced by $\alpha_n$ and $ \hat{C} = 2  C $ (for instance). This can be easily achieved again through standard mollification arguments, the construction possibly depending on $ R $ and $ \delta $. Then we set
$$
a_n:= \alpha_n + \frac {1}{n+1} \, ,
$$ 
which still guarantees \eqref{n20} (just add $ 1 $ to $ \hat{C} $). By virtue of \eqref{n5}, \eqref{j_n control-2}, the triangle inequality and the fact that $ a_n \ge 1/(n+1) $, we readily obtain 
	\begin{equation*}\label{n49}
		\left(\int_0^T \int_{B_R \setminus B_\delta } \frac{(a-a_n)^2}{a_n} \, dx dt \right)^{\frac 1 2} \leq \frac{1+\sqrt{ \left| B_R  \right| T} }{\sqrt{n+1}} \, . 
	\end{equation*}
	Therefore, thanks to \eqref{ine-estimate}, we infer that 
	\begin{equation}\label{n53}
		\lim_{n\to \infty} \limsup_{\varepsilon \to 0} I_{n,\varepsilon} = 0\,.
	\end{equation}
	Finally, letting (in this order) $ \varepsilon \to 0 $, $ n \to \infty $, $ R \to +\infty $ and $ \delta \to 0 $ in \eqref{mainuniqest}, in view of \eqref{error result}, \eqref{jn-estimate}, \eqref{j_n control}, \eqref{n53} we end up with $\int_{\R^N} \left[ u(T) - v(T) \right]  \omega \, \rho \, dx = 0 $, thus  $u(T)=v(T)$ since $ \omega $ is arbitrary. 
	\end{proof}
	
\begin{oss}\label{comparison}\rm
It is not difficult to check that the proof of Proposition \ref{prop-uniq} works under milder hypotheses. For instance, the same method yields a global \emph{comparison principle} for sub/supersolutions to \eqref{wpme} complying with \eqref{e20z}, even without continuity assumptions in $ L^1_{\mathrm{loc}}\!\left( \R^N , \rho \right) $. Recall that a (very weak) supersolution [subsolution] is a locally bounded function $ u $ satisfying
	\begin{equation*}\label{q50-subsup}
		-\int_0^T \int_{\mathbb{R}^N} u \, \phi_t \, \rho  \, dx dt \underset{ \raisebox{-0.25cm}{$[\le]$} }{\ge}  \int_0^T \int_{\mathbb{R}^N} u^m \, \Delta \phi \, dx dt + \int_{\R^N} u_0 \, \phi(0) \, \rho \, dx
	\end{equation*}
	for all nonnegative $\phi\in C^\infty_c\!\left(\mathbb{R}^N\times [0, T)\right)$. 
\end{oss}	

\begin{lem}\label{bounded-data}
Let the hypotheses of Theorem \ref{Existence} hold, and assume in addition that $ u_0 \in X_\infty $. Let $ r \ge 1 $. Then there exists a positive constant $C_6$ depending only on $ N,m,\gamma,k,K,r$ such that 
\begin{equation}\label{lemma-infty}
\left|u(x,t)\right|  \leq C_6 \left\| u_0 \right\|_{\infty,r} \left(1 + |x| \right)^{\frac{2-\gamma}{m-1}} \qquad \text{for a.e. } (x,t) \in \R^N \times (0, T_r(u_0)) \, .  
\end{equation}
\end{lem}
\begin{proof}
We construct a 	supersolution $ \overline{u} $ of the form 
$$
\overline{u}(x,t) := A \left( 1-\frac{t}{S} \right)^{-\frac{1}{m-1}} \left( 1 + |x|^2 \right)^{\frac{2-\gamma}{2(m-1)} } ,
$$
for suitable positive constants $ A $ and $S$ to be selected. Straightforward computations show that 
\begin{equation}\label{id-1}
\overline{u}_t = \frac{A}{(m-1)S} \left( 1-\frac{t}{S} \right)^{-\frac{m}{m-1}} \left( 1 + |x|^2 \right)^{\frac{2-\gamma}{2(m-1)}} , 
\end{equation}
\begin{equation}\label{id-2}
 \Delta \overline{u}^m = A^m \left( 1-\frac{t}{S} \right)^{-\frac{m}{m-1}} \frac{(2-\gamma)m}{m-1} \left[ N + \frac{(2-\gamma m ) \left| x \right|^2 }{(m-1)\left( 1+|x|^2 \right)} \right] \left( 1+|x|^2 \right)^{\frac{2-\gamma m}{2(m-1)}} .
\end{equation}
Thanks to \eqref{weight-cond} and the fact that $ \overline{u}_t \ge 0 $, the supersolution inequality
$$
\rho \, \overline{u}_t  \ge  \Delta \overline{u}^m \qquad \text{in } \R^N \times (0,S)
$$
holds provided
$$
	k \left( 1 + |x|  \right)^{-\gamma}  \, \overline{u}_t  \ge  \Delta \overline{u}^m \qquad \text{in } \R^N \times (0,S) \, ,
$$
and in view of \eqref{id-1}--\eqref{id-2} the latter is satisfied if, for instance, 
$$
S = \frac{C}{A^{m-1}}
$$
for a positive constant $ C $ depending on $ N,m,\gamma,k $. On the other hand, from the definition of $ \| \cdot \|_{\infty,r} $, it is not difficult to check that there exists $ \kappa>0 $, which depends on $ m,\gamma,r $ only,  such that 
\begin{equation}\label{id-3}
\left| u_0(x) \right|  \le \kappa \left\| u_0 \right\|_{\infty,r}  \left( 1 + |x|^2 \right)^{\frac{2-\gamma}{2(m-1)} }  \qquad \text{for a.e. } x  \in \R^N \, ,
\end{equation}
so  the choice 	$ A = \kappa \left\| u_0 \right\|_{\infty,r}   $ ensures that $ \left| u_0(x) \right| \le \overline{u}(x,0)   $ almost everywhere. By standard comparison at the level of constructed solutions (see e.g.~\cite[Subsection 3.2]{GMPjmpa}),  also noticing that $ -\overline{u} $ is a subsolution, we can therefore deduce that 
$$
\begin{gathered}
\left| u(x,t) \right| \le \overline{u}(x,t) \le 2^{\frac{1}{m-1}}  \kappa \left\| u_0 \right\|_{\infty,r} \left( 1 + |x|^2 \right)^{\frac{2-\gamma}{2(m-1)} } \le 2^{\frac{1}{m-1}}  \kappa \left\| u_0 \right\|_{\infty,r} \left( 1 + |x| \right)^{\frac{2-\gamma}{m-1} } \\ \text{for a.e. } (x,t) \in \R^N \times \left( 0,  S_r(u_0) \wedge T_r(u_0) \right) , \quad  S_r(u_0) := \tfrac {C}{2 \kappa^{m-1} \left\| u_0 \right\|_{\infty,r}^{m-1} } \,  .
\end{gathered}
$$
If $ S_r(u_0) \ge T_r(u_0) $ the proof is complete, otherwise we observe that the smoothing effect \eqref{key estimate} guarantees that
	\begin{equation*}\label{key estimate-bound}
	\begin{aligned}
\left\| u(t) \right\|_{\infty,r} \leq C_3 \, t^{-\lambda_1} \left\| u_0 \right\|_{1,r}^{ \theta \lambda_1} \le C_3 \left[ S_r(u_0) \right]^{-\lambda_1} \left\| u_0 \right\|_{1,r}^{ \theta \lambda_1} \le & \, C_6 \left\| u_0 \right\|_{\infty,r}^{\lambda_1(m-1)}  \left\| u_0 \right\|_{1,r}^{ \theta \lambda_1} \\
 \le & \, C_6 \left\| u_0 \right\|_{\infty,r} \quad \forall t \in \left( S_r(u_0) , T_r(u_0) \right) ,
\end{aligned}
\end{equation*}
where $ C_6 $ is a general positive constant as in the statement possibly changing between different computations, and in the last passage we have used the inequality $ \left\| u_0 \right\|_{1,r} \le c \left\| u_0 \right\|_{\infty,r}  $, valid for some constant $ c>0 $ depending on $ N,\gamma,K $. Thus we obtain \eqref{lemma-infty} by reasoning as in \eqref{id-3}. 
\end{proof}

We are now in position to prove uniqueness for general solutions. In order to do it, we need to carefully remove the extra pointwise bounds \eqref{e20z}, using the stability and continuity properties of the solutions constructed in Theorem \ref{Existence}.

\begin{proof}[Proof of Theorem \ref{thmuniq}]
For every $ \epsilon>0 $ (small enough) let us define $u_\epsilon $ to be the solution, provided by Theorem \ref{Existence}, whose initial datum is ${u}(\epsilon)$. Since $ u(\epsilon) \in X_\infty $, thanks to Proposition \ref{prop-uniq}, Lemma \ref{bounded-data} and the left assumption in \eqref{thmuniq-hp}, we have 
\begin{equation}\label{epsilonequal}
		 u_\epsilon(t) =  u(t+\epsilon)  \qquad \forall t \in \left( 0 , \left( S-\epsilon \right) \wedge T_r(u(\epsilon)) \right) ,
\end{equation}
 where $ r \ge 1 $ is fixed once for all and $ S \in (0,T) $ is arbitrary. Due to the definition of $ T_r(\cdot) $ and the fact that $ u \in L^\infty((0,S);X) $, we notice that 
 $$
 T_r({u}(\epsilon)) \ge  \frac{C_1}{ \left( \left\| u \right\|_{L^\infty((0,S);X) } \vee  \left\| v \right\|_{L^\infty((0,S);X) } \right)^{m-1} }  =: T^\ast \, . 
$$
If we call $ \hat{u} $ the solution of problem \eqref{wpme} constructed in Theorem \ref{Existence} (starting from $ u_0 $), then estimate \eqref{dependence on data 1} ensures that
		\begin{equation}\label{L1}
\left\| \hat{u}(t)-u_\epsilon(t) \right\|_{L^1(\Phi_\alpha)} \leq \exp\!\left({C_4 \, t^{\theta\lambda_1}}\right) \left\| u_0-u(\epsilon) \right\|_{L^1(\Phi_\alpha)} \qquad \forall t \in \left( 0 , T^\ast \right) , 
\end{equation}
where $ \alpha $ is any exponent complying with \eqref{cond-alpha} and $ C_4 $ depends only on the quantities $ N,m,\gamma,k,K,r,\alpha,\left\| u \right\|_{L^\infty((0,S);X) }$. Our next goal is therefore to prove that ${u}(\epsilon) \to u_0$ in $L^1\!\left(\Phi_\alpha\right)$ as $ \epsilon \to 0 $. To this end, first of all we use \eqref{epsilonequal} and \eqref{key estimate} to get
	\begin{equation}\label{epsilon-estimate}
	 \left| u_\epsilon(x,t) \right|^{m-1} \leq C \, \frac{\left(1+|x|^2\right)^{\frac{2-\gamma}{2}} }{t^{\lambda_1(m-1)}} \qquad \text{for a.e. } (x,t) \in  \R^N \times \left( 0 , T^\ast \right) ,
	\end{equation}
	where the positive constant $ C $ depends only on $ N,m,\gamma,k,K,r,\left\| u \right\|_{L^\infty((0,S);X) } $ (in particular is independent of $ \epsilon $).
Arguing as in the proof of \eqref{sign trick-2}, for all $\psi\in C^\infty_c\!\left(\mathbb{R}^N\right)$ with $\psi\geq0$ we have
	\begin{equation*}\label{epsilon-estimate-3}
		\frac{d}{dt}\int_{\mathbb{R}^N} \psi \left|u_\epsilon(t)\right| \rho \, dx \leq \int_{\mathbb{R}^N} \left( \Delta\psi \right) \left|u_\epsilon(t)\right|^m dx \,  ,
	\end{equation*}	
whose integration in time combined with \eqref{epsilon-estimate} entails
	\begin{equation}\label{epsilon-estimate-4}
	\begin{aligned}
& \, \int_{\mathbb{R}^N} \psi \left|u_\epsilon(t)\right| \rho \, dx \\ \leq & \, \int_{\mathbb{R}^N} \psi \left|u(\epsilon)\right| \rho \, dx +  \int_0^t \int_{\mathbb{R}^N} \left( \Delta\psi \right) \left|u_\epsilon(t)\right|^m dx  \\
\leq & \, \int_{\mathbb{R}^N} \psi \left|u(\epsilon)\right| \rho \, dx + C \, \int_0^t \frac{1}{\tau^{\lambda_1(m-1)}} \int_{\mathbb{R}^N} \left| \Delta\psi \right| \left(1+|x|^2\right)^{\frac{2-\gamma}{2}} \left|u_\epsilon(\tau)\right| dx d\tau  \quad \forall t \in \left( 0  , T^\ast \right) . 
\end{aligned}
	\end{equation}		
More rigorously, one proves \eqref{epsilon-estimate-4} on approximate solutions and then passes to the limit. Letting $ \epsilon \to 0 $, thanks to \eqref{epsilonequal}, the continuity of $ u $ with values in $ L^1_{\mathrm{loc}}\!\left( \R^N , \rho \right) $ and the uniform boundedness of the norm $ \| u_\epsilon \|_{1,r} $,  we deduce that
	\begin{equation*}\label{epsilon-estimate-5}
	\begin{aligned}
& \, \int_{\mathbb{R}^N} \psi \left|u(t)\right| \rho \, dx \\
\leq & \, \int_{\mathbb{R}^N} \psi \left|u_0\right| \rho \, dx + C \, \int_0^t \frac{1}{\tau^{\lambda_1(m-1)}} \int_{\mathbb{R}^N} \left| \Delta\psi \right| \left(1+|x|^2\right)^{\frac{2-\gamma}{2}} \left|u(\tau)\right| dx d\tau  \qquad \forall t \in \left( 0  ,S \wedge T^\ast \right) . 
\end{aligned}
	\end{equation*}		
Now we set $ \psi = \Phi_\alpha \, \phi_R $, where $ \{ \phi_R \} $ are the usual cut-off functions as in \eqref{m1}--\eqref{m52} and $\alpha $ fulfills \eqref{cond-alpha}. Due to \eqref{lap-phi}, the product formula for the Laplacian and \eqref{weight-cond}, it is not difficult to verify that there exists $ c>0 $, depending only on $ N,k,\alpha $ (in particular not on $R$), such that 
\begin{equation*}\label{delta-cutoff}
		\left|\Delta \! \left( \Phi_\alpha \, \phi_R \right) \!(x)\right| \leq \frac{ck}{2^{\frac \gamma 2}} \, \frac{\Phi_\alpha(x)}{1+|x|^2} \le c \, \frac{\Phi_\alpha(x) \, \rho(x)}{ \left(1+|x|^2\right)^{\frac{2-\gamma}{2}}}  \qquad \text{for a.e. } x \in \mathbb{R}^N \, .
	\end{equation*}
Taking the limit as $R\to+\infty$, we get
	\begin{equation*}\label{epsilon-estimate-6}
	\begin{aligned}
& \, \int_{\mathbb{R}^N}  \left|u(t)\right| \Phi_\alpha \, \rho \, dx \\
\leq & \, \int_{\mathbb{R}^N} \left|u_0\right| \Phi_\alpha \, \rho \, dx + C \, \int_0^t \frac{1}{\tau^{\lambda_1(m-1)}} \int_{\mathbb{R}^N}  \left|u(\tau)\right| \Phi_\alpha\, \rho \, dx d\tau  \qquad \forall t \in \left( 0  ,S \wedge T^\ast \right) ,
\end{aligned}
	\end{equation*}	
for another constant $ C>0 $ as above which in addition depends on $c$. By virtue of Proposition \ref{technical}(1), this implies
	\begin{equation}\label{epsilon-estimate-7}
\int_{\mathbb{R}^N}  \left|u(t)\right| \Phi_\alpha \, \rho \, dx \\
\leq \int_{\mathbb{R}^N} \left|u_0\right| \Phi_\alpha \, \rho \, dx + C E \left\| u \right\|_{L^\infty((0,S);X) } \, \int_0^t \frac{1}{\tau^{\lambda_1(m-1)}} d\tau  \quad \forall t \in \left( 0  ,S \wedge T^\ast \right) ,
	\end{equation}
	where $ E $ is the continuous-embedding constant of $ X $ into $ L^1\!\left(\Phi_\alpha \right) $. Letting $ t \to 0 $ in \eqref{epsilon-estimate-7}, we obtain 
	\begin{equation*}\label{epsilon-estimate-8}
\limsup_{t \to 0} \int_{\mathbb{R}^N}  \left|u(t)\right| \Phi_\alpha \, \rho \, dx \le \int_{\mathbb{R}^N} \left|u_0\right| \Phi_\alpha \, \rho \, dx \, .
	\end{equation*}
Thus, since $ u(t) $ converges to $ u_0 $ in $ L^1_{\mathrm{loc}}\!\left( \R^N , \rho \right) $, for every $ R >0 $ we infer that
$$
\begin{aligned}
\limsup_{t \to 0} \int_{B_R^c}  \left|u(t)\right| \Phi_\alpha \, \rho \, dx  = & \, \limsup_{t \to 0}  \left( \int_{\R^N}  \left|u(t)\right| \Phi_\alpha \, \rho \, dx -  \int_{B_R}  \left|u(t)\right| \Phi_\alpha \, \rho \, dx  \right)  \\
\le & \, \int_{\mathbb{R}^N} \left|u_0\right| \Phi_\alpha \, \rho \, dx -  \int_{B_R}  \left|u_0\right| \Phi_\alpha \, \rho \, dx  = \int_{B_R^c}  \left|u_0\right| \Phi_\alpha \, \rho \, dx \, ,
\end{aligned}
$$
whence
$$
\begin{aligned}
& \,\limsup_{t \to 0} \int_{\R^N}  \left|u(t)-u_0\right| \Phi_\alpha \, \rho \, dx \\
\le & \, \limsup_{t \to 0} \int_{B_R}  \left|u(t)-u_0\right| \Phi_\alpha \, \rho \, dx + \limsup_{t \to 0} \int_{B_R^c}  \left|u(t)\right| \Phi_\alpha \, \rho \, dx + \int_{B_R^c}  \left|u_0\right| \Phi_\alpha \, \rho \, dx \\
\le & \, 2 \int_{B_R^c}  \left|u_0\right| \Phi_\alpha \, \rho \, dx \, , 
\end{aligned}
$$
so we reach the desired $ L^1\!\left( \Phi_\alpha \right) $ convergence by finally letting $ R \to +\infty $. As a result, taking the limit of \eqref{L1} as $\epsilon \to 0 $ and recalling \eqref{epsilonequal}, we find that $ \hat{u}(t) = u(t) $ for every $ t \in (0,S \wedge T^\ast) $. Repeating exactly the same proof with $ v $ replacing $ u $ yields $ v(t)=\hat{u}(t)=u(t) $ for every $ t \in (0,S \wedge T^\ast) $. If $ T^\ast<S $ then we can apply Proposition \ref{prop-uniq} to extend uniqueness to the time interval $ (T^\ast,S) $. In any case, we obtain $ u(t)=v(t) $ for all $ t \in (0,S) $, which implies the thesis since $ S $ can be taken arbitrarily close to $T$.
\end{proof}

\begin{oss}\label{N=2} \rm 
	{As mentioned in the Introduction,  we require throughout the paper that $ N \ge 3 $. However, most of our results also hold when $ N=2 $ or $N=1$, up to some clarifications. In particular, in the proof of Proposition \ref{pro2} one should exploit suitable weighted Gagliardo-Nirenberg inequalities (see \cite{CKN} at least for $ \rho(x) = |x|^{-\gamma} $) instead of the weighted Sobolev inequality \eqref{e12}, and argue as in \cite[pp.~67--69]{BCP}. Note that this forces the further restriction $ \gamma<1 $ if $ N=1 $, which is in fact also necessary for $ \rho $ to be locally integrable. In addition, in the proof of Proposition \ref{prop-uniq}, the $N$-dimensional Green function used in \eqref{barrier-2} should be replaced by the $2$-dimensional logarithmic one if $N=2$, whereas a slightly different approach is required if $ N=1 $ (the boundary integral on $ \partial B_\delta $ does not vanish anymore as $ \delta \to 0 $). In this case, it is more convenient to solve the dual problem \eqref{n9} with the approximate weight $ \rho_n $ replaced by  the original one $ \rho$ and in the whole ball $ B_R $, some extra care being required due to the fact that the dual solution $ \xi_n $ is no more a classical one.}
\end{oss}



\section{Blow-up}\label{blowup}

In order to construct radial solutions to the semilinear elliptic equation \eqref{ellip-base}, first of all we take advantage of ODE shooting techniques to solve the following Cauchy problems:
\begin{equation}\label{ode1}
	\begin{cases}
 \left(w^m\right)''(y) + \frac{N-1}{y} \left(w^m\right)'(y) = \rho(y) \, w(y) & \text{for a.e. } y>0 \, , \\
		w(0) = \beta \, , \\
	\left(	w^m \right)'( y ) = o\!\left(y^{N-1}\right) & \text{as } y \to 0 \, ,
	\end{cases}
\end{equation}
and then we analyze the asymptotic behavior as $ y \to + \infty $. Note that, with abuse of notation, by $ \rho(y) $ we mean the one-variable function representing the radial function $ \rho(x) $. The initial condition on the radial derivative in \eqref{ode1} replaces the classical one because for $ \gamma \in [1,2) $, with the singular weight $ y^{-\gamma} $, it is not possible to require that $ \left(	w^m \right)'( 0 ) = 0  $.

\begin{proof}[Proof of Proposition \ref{sol-ellip}]

We employ a strategy inspired by \cite[Lemma 5.5]{GMPjmpa}, and divide the proof into three main steps. By means of a routine scaling argument, it is enough to prove all of the results for the case $ T=1/(m-1) $, where the equation becomes \eqref{ellip-base}, which will be tacitly assumed from now on. 
	
\medskip 	
	
\noindent \textsc{Step 1} (Equivalent problems and existence). 

If we let $ V := w^m $, then the ODE underlying \eqref{ellip-base} can be written as 
	\begin{equation}\label{var-ode1}
		\left(y^{N-1} \, V'(y) \right)' = y^{N-1} \, \rho(y) \,  V(y)^{\frac 1 m } \, , 
	\end{equation}
	for $ y >0$. Therefore, since the nonlinearity on the RHS is sublinear, global existence is guaranteed as long as local existence is. Given the assumptions on $ \rho $, it is not difficult to check that the Cauchy problem \eqref{ode1} is actually equivalent to the integral formulation 
		\begin{equation}\label{var-ode1-int}
		V(y) = \beta^m + \int_0^y  \frac{1}{z^{N-1}}\, \int_0^z \zeta^{N-1} \, \rho(\zeta) \,  V(\zeta)^{\frac 1 m }  \, d\zeta   dz \qquad \forall y \ge 0 \, .
	\end{equation}
	On the other hand, a local solution to \eqref{var-ode1-int} can be obtained through a standard fixed point method, working in the complete metric space
	$$
	\left\{ V \in C^0\!\left( [0,\varepsilon] \right): \ \beta^m \le V(y) \le \mu \ \ \forall y \in [0,\varepsilon]  \right\} ,
	$$
where $ \mu,\varepsilon>0 $ are suitable constants depending on $ N,\gamma,K,\beta $. Thus, a (global) solution to \eqref{ode1} exists, and \eqref{var-ode1-int} ensures that it is strictly increasing, a property that we will now exploit.  	
	
\medskip 	
	
\noindent \textsc{Step 2} (Asymptotic behavior). 

From \eqref{var-ode1-int}, \eqref{weight-cond} and the monotonicity of $V$ we get 
\begin{equation*}\label{sol-ellliptic-radial-integ-bis}
V^\prime(y) = \frac{\int_0^y z^{N-1} \, \rho(z) \,  V(z)^{\frac 1 m }  \, dz}{y^{N-1}} \le \frac{K}{N-\gamma} \, y^{1-\gamma} \, V(y)^{\frac{1}{m}}  \qquad \forall  y >0 \, ,
\end{equation*}
which is equivalent to
\begin{equation*}\label{sol-ellliptic-radial-integ-ter}
\left(V^{\frac{m-1}{m}} \right)^{\prime}(y) \le \frac{m-1}{m}\,\frac{K}{N-\gamma} \, y^{1-\gamma}   \qquad \forall  y >0 \, ,
\end{equation*}
whose integration readily yields
\begin{equation}\label{sol-ellliptic-radial-integ-qqs}
\limsup_{y \to +\infty} \frac{V(y)}{y^{\frac{2-\gamma}{m-1}m}} \le \left[ \frac{(m-1)K}{m(N-\gamma)(2-\gamma)} \right]^{\frac{m}{m-1}} .
\end{equation}
Next, we prove the reverse bound (for the $ \liminf $) by induction. To this purpose, let us set
$$
F(y) := 
\begin{cases}
y^2 & \text{if } y \in [0,1] \, ,  \\
y^{2-\gamma} & \text{if } y > 1 \, .
\end{cases}
$$
 For a given $ n \in \mathbb{N} $, assume that $ V $ satisfies
\begin{equation}\label{V-liminf-1}
V(y) \ge c_n \, F (y)^{p_n} \qquad \forall y \ge 0 \, ,
\end{equation}
where $ \{ c_n \} $ and $ \{ p_n \} $ are suitable nonnegative sequences to be chosen, with $ \{ p_n \} $  increasing and less than $ m/(m-1) $. Under these hypotheses, a straightforward computation shows that there exists a positive constant $ C>0 $, depending only on $ N$ and $\gamma $, such that   
\begin{equation}\label{V-liminf-2}
\int_0^y  \frac{1}{z^{N-1}}\, \int_0^z \zeta^{N-1} \, (1+\zeta)^{-\gamma} \,  F(\zeta )^q  \, d\zeta   dz  \ge  \frac{C}{\left( q+1 \right)^2} \, F(y)^{q+1}   \qquad \forall y \ge 0 \, , \ \forall q \ge 0 \, . 
\end{equation} 
Substituting \eqref{V-liminf-1} into \eqref{var-ode1-int}, and using \eqref{V-liminf-2} with $ q=p_n/m $ and the lower bound in  \eqref{weight-cond}, we obtain 
\begin{equation*}\label{V-liminf-3}
\begin{gathered}
		V(y) \ge c^{\frac 1 m}_n  k \, \int_0^y  \frac{1}{z^{N-1}}\, \int_0^z \zeta^{N-1} \, (1+\zeta)^{-\gamma} \,  F(\zeta)^{\frac {p_n}{m} }  \, d\zeta   dz \ge  \frac{k\,C}{\left( \frac{p_n}{m}+1 \right)^2}  \, c^{\frac 1 m}_n  F(y)^{\frac{p_n}{m}+1} \\[0.15cm]
		 \forall y \ge 0 \, .
\end{gathered}
\end{equation*} 
Note that, up to allowing $C $ to depend on $ m $ and $ k $ as well, the above inequality can be rewritten as 
\begin{equation*}\label{V-liminf-4}
V(y) \ge C \, c_n^{\frac 1 m } F(y)^{\frac{p_n}{m}+1} \qquad \forall y \ge 0 \, .  
\end{equation*}
Thus, if we recursively choose 
\begin{equation}\label{V-liminf-5}
p_{n+1} = \frac{p_n}{m}+1 \, , \qquad c_{n+1} = C \, c^{\frac 1 m}_n , 
\end{equation}
and observe that \eqref{V-liminf-1} trivially holds with $ p_0 =0 $ and $ c_0 = \beta $, by induction we deduce that \eqref{V-liminf-1} is satisfied for every $ n \in \mathbb{N}  $. As \eqref{V-liminf-5} can be solved explicitly, we end up with 
\begin{equation*}\label{V-liminf-6}
V(y) \ge \beta^{\frac{1}{m^{n}}}  \, C^{\frac{m-m^{1-n}}{m-1}} \,  F (y)^{\frac{m}{m-1}\left( 1-\frac{1}{m^n} \right) } \qquad \forall y \ge 0 \, , \ \forall n \in \mathbb{N} \, ,
\end{equation*}
which implies, by letting $ n \to \infty $, 
\begin{equation*}\label{V-liminf-7}
V(y) \ge C^{\frac{m}{m-1}} \,  F (y)^{\frac{m}{m-1}} \qquad \forall y \ge 0 \, . 
\end{equation*}
As a result, recalling the definition of $F$, we finally infer that 
\begin{equation}\label{sol-ellliptic-radial-integ-liminf}
\liminf_{y \to +\infty} \frac{V(y)}{y^{\frac{2-\gamma}{m-1}m}} \ge C^{\frac{m}{m-1}} \, .
\end{equation}
An elementary computation shows that 
$$
\underline{\kappa} \, \liminf_{y \to +\infty}  \frac{w(y)}{y^{\frac{2-\gamma}{m-1}}} \le \lim_{r \to +\infty}  \left\| w \right\|_{1,r} \le  \overline{\kappa} \, \limsup_{y \to +\infty} \frac{w(y)}{y^{\frac{2-\gamma}{m-1}}}  \, , 
$$
where $ \underline{\kappa} , \overline{\kappa} $ are suitable positive constants depending only on $ N,m,\gamma,k,K $. Hence, thanks to \eqref{sol-ellliptic-radial-integ-qqs} and \eqref{sol-ellliptic-radial-integ-liminf}, estimate \eqref{ellip-ord-2} (when $ T=1/(m-1) $) is proved.  
	
\medskip 	
	
\noindent \textsc{Step 3} (Ordering principle). 

Let $ w_1 $ and $ w_2 $ denote the two solutions of \eqref{ode1} corresponding to $ \beta = \beta_2 $ and $ \beta = \beta_1 $, respectively, with $ \beta_2 > \beta_1 >0 $. Subtracting the equations in the form \eqref{var-ode1} and integrating we obtain 
$$
\begin{gathered}
y^{N-1} \left( w_2^m - w_1^m \right)'(y) - z^{N-1} \left( w_2^m - w_1^m \right)'(z)   = \int_z^y \zeta^{N-1} \, \rho(\zeta) \left( w_2(\zeta)-w_1(\zeta) \right) d \zeta  \\ \forall y,z>0 : \ y>z \, , 
\end{gathered}
$$
whence, by letting $ z \to 0 $ and using the initial condition on the derivatives, 
\begin{equation}\label{eq-comparison}
y^{N-1} \left( w_2^m - w_1^m \right)'(y)  = \int_0^y \zeta^{N-1} \, \rho(\zeta) \left( w_2(\zeta)-w_1(\zeta) \right) d \zeta  \qquad \forall y>0 \, . 
\end{equation}
Let
$$
\bar{y} := \sup \left\{ y>0 : \ w_2(z)>w_1(z) \ \, \text{for all } z <y  \right\} .
$$
If, by contradiction, $ \bar{y} < + \infty $ then $ w_2(\bar{y}) = w_1(\bar{y}) $ and $ w_2(z)>w_1(z) $ for all  $ z \in (0,\bar{y}) $, but this is clearly incompatible with \eqref{eq-comparison}, which implies the monotonicity of $ \left( w_2^m - w_1^m \right) $ in the interval $  (0,\bar{y}) $. 
\end{proof}

Using $ \{ W_\beta \} $ via the separable solutions \eqref{def u2}, we can now prove our main blow-up result, by proceeding similarly to \cite[Theorem 2.7]{GMPjmpa}.

\begin{proof}[Proof of Theorem \ref{blowupthm}]
For an initial datum $ u_0 $ satisfying \eqref{blowupthmeq}, thanks to Proposition \ref{weight-cond} it is clear that $ u_0 \in X_\infty $. Moreover, the two-sided bound \eqref{ellip-ord-3} is a trivial consequence of \eqref{ellip-ord} evaluated both at $ \beta=\beta_1 $ and $\beta = \beta_2 $, along with the monotonicity of $ \ell(\cdot) $ with respect to the argument. In view of Lemma \ref{bounded-data} and Remark \ref{comparison}, we can therefore assert that 
\begin{equation}\label{estim-comparison}
U_{\beta_1}(x,t) \le u(x,t) \le U_{\beta_2}(x,t) \qquad \text{for a.e.~} (x,t) \in \left( 0 , T(u_0) \wedge T \right) . 
\end{equation}
On the other hand, estimate \eqref{estim-comparison} itself guarantees that actually the existence time of $ u $ can be pushed up to $ t=T $; indeed, if $ T(u_0) < T $ one can restart the construction procedure of Theorem \ref{Existence} using as a new initial datum $ u(T(u_0)) $, which is less than $ U_{\beta_2}(T(u_0))  $ and thus still belongs to $ X_\infty $. This can be carried out in a finite number of steps up to an arbitrary $ S < T $. As a result, we deduce that in fact 
\begin{equation*}\label{estim-comparison-2}
U_{\beta_1}(x,t) \le u(x,t) \le U_{\beta_2}(x,t) \qquad \text{for a.e.~} (x,t) \in \left( 0 , T \right) ,
\end{equation*}
which readily yields \eqref{ptwse-blowup}. Note that, for this special class of initial data, the construction of the solution to \eqref{wpme} could have been alternatively obtained by an \emph{ad hoc} procedure consisting of solving Dirichlet problems on balls $ B_R $ (see e.g.~\cite[Theorem 2.2]{GMPjmpa}) prescribing as boundary conditions either $ U_{\beta_1} $ or $ U_{\beta_2} $, and taking limits as $ R \to +\infty $. The advantage is that only a local comparison is needed and the constructed solution is naturally defined in the whole time interval $ (0,T) $.
\end{proof}

\appendix
\section{Some technical results}\label{global in time}
We collect here a few useful facts concerning the spaces $ X_p $ introduced in Subsection \ref{sf}. 
\begin{pro}\label{technical}
	The following hold:
	\begin{enumerate}
		\item For all $\alpha>\frac{2-\gamma}{2(m-1)}+\frac{N-\gamma}{2}$, we have $X \hookrightarrow L^1\!\left(\Phi_\alpha\right)$;  \label{A}
		\item If $ \left\{ f_n \right\} \subset L^1_{\mathrm{loc}}\!\left(\mathbb{R}^N,\rho\right)$ and $f_n\to f$ in $ L^1_{\mathrm{loc}}\!\left(\mathbb{R}^N,\rho\right)$, then 
		\begin{equation*}
			\left\| f \right\|_{p,r}\leq \liminf_{n\to\infty} \left\| f_n \right\|_{p,r} \qquad \forall p \in [1,\infty] \, , \ \forall r \ge 1 \, ,
		\end{equation*}
		and the same inequality holds for the norm $ |\cdot|_{p,r} $. \label{B}
	\end{enumerate}
\end{pro}

\begin{proof}
	Let us prove the two properties separately. 
	
\noindent \eqref{A} Given any $f\in X$, from the definition of $ \| \cdot \|_{1,r} $ we have
	\begin{equation*}
		R^{-\frac{2-\gamma}{m-1}-N+\gamma} \, \int_{B_R} \left| f \right|\rho \, dx \leq \left\| f \right\|_{1,r} =: M  \qquad \forall R \ge r \ge 1 \, .
	\end{equation*}
	Next, we notice that 
	\begin{equation}\label{a2}
	\begin{aligned}
 \left\| f \right\|_{L^1\left( \Phi_\alpha \right)}	=	\int_{\mathbb{R}^N} \frac{\left|f(x)\right|}{\left(1+|x|^2\right)^\alpha}\, \rho(x) \, dx \leq & \, \int_{B_r} \left|f(x)\right|\rho(x)\, dx +\int_{B_r^c}\frac{\left|f(x)\right|}{|x|^{2\alpha}} \, \rho(x) \, dx \\
 \leq & \, r^{\frac{2-\gamma}{m-1}+N-\gamma} \, M + \int_{B_r^c}\frac{\left|f(x)\right|}{|x|^{2\alpha}} \, \rho(x) \, dx \, .
 \end{aligned}
	\end{equation}
	Now, let us write 
	\begin{equation*}
		\frac{1}{|x|^{2\alpha}} = 2\alpha \, \int_{|x|}^{+\infty} \frac{1}{R^{2\alpha+1}} \, dR \, ,
	\end{equation*}
	so that by Fubini's theorem 
	\begin{align*}
		\int_{B_r^c}\frac{\left|f(x)\right|}{|x|^{2\alpha}} \, \rho(x) \, dx & = 2\alpha \, \int_r^{+\infty}\frac{1}{R^{2\alpha+1}}\left(\int_{B_R\setminus B_r} \left| f(x) \right|\rho(x) \, dx\right) dR \\ & \leq2\alpha\, M \, \int_r^{+\infty} \frac{R^{\frac{2-\gamma}{m-1}+N-\gamma}}{R^{2\alpha+1}} \, dR =  \frac{2\alpha\,r^{-\left( 2\alpha-\frac{2-\gamma}{m-1}-N+\gamma \right)}}{2\alpha-\frac{2-\gamma}{m-1}-N+\gamma} \, M \, ,
	\end{align*}
	where we are using the assumption that $2\alpha>\frac{2-\gamma}{m-1}+N-\gamma$. Combining this with \eqref{a2}, we obtain the desired continuous-embedding bound.
	
\noindent \eqref{B} With no loss of generality, since $f_n\to f$ in $L^1_{\mathrm{loc}}\!\left(\mathbb{R}^N,\rho\right)$, we can suppose in addition that $ f_n \to f $ pointwise almost everywhere. Hence, thanks to Fatou's lemma, for each $ R \ge r $ and $ p \in [1,\infty) $ we have
	\begin{equation}\label{R-p}
		R^{-\frac{2-\gamma}{m-1}\, p-N+\gamma} \, \int_{B_R} \left|f\right|^p\rho\,dx \leq \liminf_{n\to\infty} R^{-\frac{2-\gamma}{m-1}\, p-N+\gamma} \, \int_{B_R} \left|f_n\right|^p\rho\,dx
		\leq \liminf_{n\to\infty} \left\| f_n \right\|_{p,r}^p .
	\end{equation}
Raising to the $  1 / p $ and taking the supremum over $R\geq r$ completes the proof. In the case $ p=\infty $, we observe that there exists a constant $  C>0$, depending only on $ N,\gamma,K $, such that $ \| \cdot \|_{p,r} \le C^{ 1 / p} \, \| \cdot \|_{\infty,r} $. Hence, from \eqref{R-p} we deduce  
$$
\frac{\left\| f \right\|_{L^p\left(B_R , \rho\right)} }{R^{\frac{2-\gamma}{m-1}+\frac{N-\gamma}{p}}} \le C^{\frac 1 p} \liminf_{n\to\infty} \left\| f_n \right\|_{\infty,r} , 
$$
whereby the claimed inequality upon letting $ p \to \infty $ and taking the supremum $R\geq r$. As for the norms $ | \cdot |_{p,r} $, the same proof follows with inessential modifications. 
%
%
\end{proof}

\begin{pro}\label{closedness}
We have
$$
X_0 = \overline{L^1\!\left(\mathbb{R}^N,\rho\right)}^X =: Y \, .
$$
Moreover, for every $ f \in X_0 $ the truncated sequence $ f_n := \tau_n(f) \, \chi_{B_n} $ converges to $ f $ in $ X $, where $ \tau_n $ is defined in \eqref{truncation-function}. 
\end{pro}
\begin{proof}
	Let us first prove the second part  of the statement, which in particular implies the inclusion $ X_0 \subseteq Y $. We claim that
	\begin{equation*}\label{claim}
\lim_{n \to \infty}	\left\| f_n - f \right\|_{1,r} = 0 \, ,
	\end{equation*} 
	for any $ r \ge 1 $. To this end,  fix $ \omega := \frac{2-\gamma}{m-1}+N-\gamma$ and $ r_0 > r $. By the construction of $ f_n $, it follows that 
	\begin{equation*}
	R^{-\omega} \, \int_{B_R} \left| f_n-f \right| \rho  \, dx \leq
	\begin{cases}
	R^{-\omega}  \, \int_{B_{r_0}} \left| f_n - f \right| \rho \, dx  & \text{if }   r \le R \le r_0 \, , \\
	 R^{-\omega}  \, \int_{B_{R}} \left|  f \right| \rho \, dx  & \text{if } R > r_0 \,  .
	 \end{cases}
	\end{equation*}
	Next, we take the supremum over $R\geq r$ on both sides. This yields
	\begin{equation*}
	\left\| f_n- f \right\|_{1,r} \leq \max\left\{  r^{-\omega} \, \int_{B_{r_0}}  \left| f_n - f \right| \rho \, dx  \,   , \, \left\| f \right\|_{1,r_0} \right\} . 
	\end{equation*}
	Then we pass to the limit as $ n \to \infty $ and apply dominated convergence on the first term on the RHS, obtaining
	\begin{equation*}
	\limsup_{n\to \infty} \left\| f_n-f \right\|_{1,r} \leq  \left\| f \right\|_{1,r_0} .
	\end{equation*}
	Finally, taking the limit as $r_0\to\infty$ we get the claim, since $ f \in X_0 $. 
	
	Now, we show the opposite inclusion $ X_0 \supseteq Y $.  Let $ g \in Y $ and $ h \in L^1\!\left(\mathbb{R}^N,\rho\right)$ be arbitrary functions. Since  
	\begin{equation*}
	R^{-\omega} \, \int_{B_R} \left| g \right| \rho \, dx \leq R^{-\omega} \, \int_{B_R} \left|h-g\right| \rho \, dx +R^{-\omega} \, \int_{B_R} \left| h \right| \rho \, dx \, ,  
	\end{equation*}
	we can take the supremum over $R\geq r \ge r_0 \ge 1$ on both sides to obtain 
	\begin{equation*}
	\left\|  g \right\|_{1,r} \leq \left\| h-g \right\|_{1,r_0}+r^{-\omega} \left\| h \right\|_{L^1\left(\mathbb{R}^N , \rho\right)} ,
	\end{equation*}
	where we also used the fact that $ r \mapsto \| \cdot \|_{1,r} $ is nonincreasing. Passing to the limit as $ r \to +\infty $ we end up with 
	\begin{equation*}
	\lim_{r \to + \infty} \left\| g \right\|_{1,r} \leq \left\| h-g \right\|_{1,r_0} . 
	\end{equation*}
	Since $ g \in Y $, by the definition of this space we can select $ h \in L^1\!\left( \mathbb{R}^N , \rho \right) $ in such a way that the RHS is arbitrarily small, so it follows that $\lim_{r\to+\infty} \left\| g \right\|_{1,r} = 0 $ as desired.  
\end{proof}

\begin{pro}\label{limsup-norm}
Let $ f \in X_\infty $. Then 
\begin{equation*}\label{limsup-identity}
\lim_{r \to +\infty} \left\| f \right\|_{\infty,r} = \underset{|x| \to +\infty}{\operatorname{ess}\limsup} \ |x|^{-\frac{2-\gamma}{m-1}}  \left| f(x) \right| .
\end{equation*}
\end{pro}
\begin{proof}
The inequality
$$
L:=\underset{|x| \to +\infty}{\operatorname{ess}\limsup} \ |x|^{-\frac{2-\gamma}{m-1}}  \left| f(x) \right| \le \limsup_{R \to +\infty}  \frac{\left\| f \right\|_{L^\infty(B_R)}}{R^{\frac {2-\gamma}{m-1}}}  \le  \left\| f \right\|_{\infty,r} \qquad \forall r \ge 1
$$
is straightforward, whence $ L \le \lim_{r \to +\infty}  \left\| f \right\|_{\infty,r}  $. In order to prove the reverse, we notice that  for all $ R_0,r \ge 1 $, with $ r>R_0 $, we have
\begin{equation}\label{ee1}
\begin{aligned}
\left\| f \right\|_{\infty,r} = & \, \max \left\{ \sup_{R \ge r} \frac{\left\| f \right\|_{L^\infty\left(B_{R_0}\right)}}{R^{\frac {2-\gamma}{m-1}}}  , \,  \sup_{R \ge r} \frac{\left\| f \right\|_{L^\infty\left(B_R \setminus B_{R_0} \right)}}{R^{\frac {2-\gamma}{m-1}}}  \right\}  \\
= & \,  \max \left\{ \frac{\left\| f \right\|_{L^\infty\left(B_{R_0}\right)}}{r^{\frac {2-\gamma}{m-1}}}  , \,  \sup_{R \ge r} \frac{\left\| f \right\|_{L^\infty\left(B_R \setminus B_{R_0} \right)}}{R^{\frac {2-\gamma}{m-1}}}  \right\}  . 
\end{aligned}
\end{equation}
From the definition of $ \operatorname{ess}\limsup $, for every $ \varepsilon>0 $ we can pick $ R_0 $ so large that 
$$
\left|f(x)\right| \le  \left( L +\varepsilon \right) \left| x \right|^{\frac{2-\gamma}{m-1}} \qquad \text{for a.e. } x \in B_{R_0}^c \, .
$$
Hence, substituting into \eqref{ee1}, we reach 
$$
\left\| f \right\|_{\infty,r}  \le \max \left\{ \frac{\left\| f \right\|_{L^\infty\left(B_{R_0}\right)}}{r^{\frac {2-\gamma}{m-1}}}  , \,  L + \varepsilon \right\}  \qquad \forall r > R_0 \, .
$$
The thesis follows by letting first $ r \to +\infty $ and then $ \varepsilon \to 0 $. 
\end{proof}

\noindent{\textbf{Acknowledgments.} M.M.~was partially funded by the PRIN 2017 project ``Direct and Inverse Problems for Partial Differential Equations: Theoretical Aspects and Applications'' (grant no.~201758MTR2, MIUR, Italy). He also thanks the GNAMPA group of INdAM (Italy). {Both authors are grateful to Fernando Quir\'os for fruitful discussions and valuable suggestions that helped improve a first draft of this paper.}}

\end{document}